\documentclass[preprint]{elsarticle}
\usepackage[ansinew]{inputenc}
\usepackage[english]{babel}
\usepackage{amsmath,verbatim,amsthm,hyperref,color,amssymb,pifont,natbib,geometry,fleqn,txfonts}
\usepackage{graphics}
\usepackage{amsfonts,enumerate}

\swapnumbers
\theoremstyle{plain}
\numberwithin{equation}{section}
\allowdisplaybreaks

\newtheorem{lemma}{Lemma}[section]
\newtheorem{proposition}[lemma]{Proposition}
\newtheorem{theorem}[lemma]{Theorem}
\newtheorem{cor}[lemma]{Corollary}
\newtheorem{remark}[lemma]{Remark}
\newtheorem{definition}[lemma]{Definition}
\newtheorem{example}[lemma]{Example}
\newtheorem{assumption}[lemma]{Assumption}
\newtheorem{notation}[lemma]{Notation}

\newcommand{\mitt}{\left | { \atop }  \right.}
\newcommand{\equa}{\begin{eqnarray*}}
\newcommand{\tion}{\end{eqnarray*}}
\newcommand{\Om}{\Omega}
\newcommand{\om}{\omega}
\newcommand{\vare}{\varepsilon}
\newcommand{\vph}{\varphi}
\newcommand{\half}{{\frac{1}{2}}}

\newcommand{\ftn}{{\mathcal{F}}}
\newcommand{\gtn}{{\mathcal{G}}}
\newcommand{\E}{{\mathbb{E}}}

\newcommand{\R}{{\mathbb{R}}}
\newcommand{\N}{{\mathbb{N}}}
\newcommand{\Q}{{\mathbb{Q}}}

\renewcommand{\P}{\mathbb{P}}
\newcommand{\cF}{{\mathcal{F}}}
\newcommand{\cL}{{\mathcal{L}}}
\newcommand{\cG}{{\mathcal{G}}}
\newcommand{\cH}{{\mathcal{H}}}
\newcommand{\cP}{{\mathcal{P}}}
\newcommand{\cN}{{\mathcal{N}}}
\newcommand{\sptext}[3]{\hspace{#1 em}\mbox{#2}\hspace{#3 em}}
\newcommand{\C}{\mathcal{C}}
\newcommand{\rh}{\mathcal{RH}}
\newcommand{\la}{\lambda}
\newcommand{\bmo}{{\rm BMO}}
\newcommand{\M}{\mathbb{X}}
\newcommand{\A}{\mathbb{A}}
\newcommand{\weight}{w}
\newcommand{\sli}{\rm sl}

\newcommand{\BMO}{{\rm BMO}}
\newcommand{\thteta}{\alpha}

\newcommand{\st}{}

\begin{document}

\begin{frontmatter}
	
	\title{Weighted Bounded Mean Oscillation applied to Backward Stochastic Differential Equations}
	\author[ALL]{Stefan Geiss}
        \ead{stefan.geiss@jyu.fi}

        \author[ALL,JY]{Juha Ylinen}
	
	\address{Department of Mathematics and Statistics,
		University of Jyv\"askyl\"a,
		P.O. Box 35,
                FIN-40014 University of Jyv\"askyl\"a,
		Finland.}
	
\fntext[ALL]{This work was supported by project "Stochastic and Harmonic Analysis, interactions, and applications" of the Academy of Finland [project number 133914].}
\fntext[JY]{{\st The author was supported by the} Vilho, Yrjö and Kalle Väisälä foundation of the Finnish Academy of Science and Letters.}

\begin{abstract}
\noindent We deduce conditional $L_p$-estimates for the variation of a solution of a BSDE. Both quadratic and sub-quadratic types of BSDEs are considered,
and using the theory of weighted bounded mean oscillation we deduce new tail estimates for the solution $(Y,Z)$ on subintervals of $[0,T]$.
Some new results for the decoupling technique introduced in \cite{jossain} are obtained as well and {\st some applications 
of the tail estimates are given.}
\end{abstract}


\end{frontmatter}

\tableofcontents


\section{Introduction}

In this article we study backward stochastic differential equations (BSDEs from now on) of type
\begin{equation}\label{equation:BSDE0}
Y_t = \xi + \int_t^T f(s,Y_s,Z_s)ds - \int_t^T Z_s dW_s, \qquad t \in [0,T],
\end{equation}
where $T>0$ is a fixed number and $(W_t)_{t \in [0,T]}$ is a $d$-dimensional Brownian motion.
Roughly speaking, a BSDE is a map $(\xi,f)\mapsto (Y,Z)$, so that $(\xi,f)$ is the data, and $(Y,Z)$ is the solution.
Here the \emph{terminal value} $\xi \in L_2$ is a given random variable that is measurable with respect to the $\sigma$-algebra generated by the Brownian motion.
In the present article, the \emph{generator} 
$f:[0,T]\times\Om\times\R\times\R^d \to \R$ is assumed to be such that
\begin{itemize}
	\item[(1)] $(t,\om) \mapsto f(t,\om,y,z)$ is predictable for all $(y,z) \in \R \times \R^{d}$, and
	\item[(2)] there are $L_y,L_z\ge 0$ and a $\theta\in [0,1]$
	such that for all $(t,\om,y_0,y_1,z_0,z_1)$ one has
	\begin{equation*}
	|f(t,\om,y_0,z_0)-f(t,\om,y_1,z_1)|
	\le L_y |y_0-y_1| + L_z [1+|z_0|+|z_1|]^\theta |z_0-z_1|.
	\end{equation*}
\end{itemize}
This means that the generator $f$ can be random, is assumed to be uniformly Lipschitz in the $y$-variable, and locally Lipschitz in the $z$-variable.
We will consider the uniformly Lipschitz case ($\theta=0$), the quadratic case ($\theta=1$), and the sub-quadratic case ($\theta\in (0,1)$) at the same time.
We say that $(Y,Z)$ is a solution of BSDE \eqref{equation:BSDE0} if $Y$ is a continuous adapted process with $\E\sup_{t \in [0,T]} |Y_t|^2 < \infty$,
if $Z$ is a predictable process with $\E \int_0^T |Z_r|^2 dr < \infty$, and if \eqref{equation:BSDE0} is satisfied almost surely.
\parindent=20pt

BSDEs were first introduced by Bismut in \cite{Bismut}, and the amount of research increased significantly after Pardoux and Peng showed in \cite{PP1} that a BSDE with square-integrable terminal value $\xi$ and a uniformly Lipschitz generator $f$ has a unique solution. Concerning the Lipschitz-case, see also for example \cite{PP2}, \cite{ElKaroui:97}, and \cite{Briand:02}. More recently, the theory of BSDEs with a generator that grows quadratically in the $z$-variable has been developed,
see for example \cite{Kobylanski:00}, \cite{HIM}, \cite{Morlais}, \cite{Briand:12}, \cite{DHR}, and the references therein.
The original motivation of studying BSDEs comes from stochastic optimal control theory. {\st In general,} BSDEs have applications in stochastic differential games,
stochastic finance in connection to option pricing
and utility maximization, and they are closely connected to partial differential equations (PDEs).
\pagebreak

{\st
The present article is a continuation of \cite{jossain}  and an application of \cite{Geiss:wBMO},
where \cite{jossain} itself is a continuation  of \cite{GGG:12}. The main results of this article are
\smallskip

\begin{itemize}
\setlength\itemsep{-0.1em}
\item[--]
{\sc Theorem \ref{theorem:main}}, that provides conditional variational estimates for the processes $(Y,Z)$,
      i.e. we bound the {\em mean oscillation} of the processes $(Y,Z)$ from above
      by natural {\em weights} derived from the initial data $(\xi,f)$  of the BSDE, 
\item[--]
 {\sc Theorem \ref{theorem:tail-estimate}}, that deduces from Theorem \ref{theorem:main} 
      {\em conditional tail estimates of John-Nireneberg type},
\item[--]
{\sc Theorem \ref{theorem:main_FBSDE}}, that is the version of Theorem \ref{theorem:main} for decoupled FBSDEs,
\item[--]
 {\sc Theorems \ref{theorem:FBSDE_tail_Y} and \ref{theorem:FBSDE_tail_Z}}, that are versions of Theorem \ref{theorem:tail-estimate}
      for decoupled FBSDEs.
\end{itemize}

Our strategy to prove the basic 
Theorems \ref{theorem:main} and \ref{theorem:tail-estimate} consists in the following steps:

\begin{itemize}\setlength\itemsep{-0.1em}
\item[--] {\sc Step} 1: We prove a conditional decoupling inequality for BSDEs in Proposition \ref{proposition:I_1}.
\item[--] {\sc Step} 2: We deduce conditional variational inequalities for $(Y,Z)$ in Theorem \ref{theorem:main}.
\item[--] {\sc Step} 3: We deduce conditional tail estimates for $(Y,Z)$ in Theorem \ref{theorem:tail-estimate}.
\end{itemize}

In Steps 1 and 2 we extend and apply methods from \cite{jossain},
in Step 3 we use a result from \cite{Geiss:wBMO}.
To explain the role of \cite{Geiss:wBMO} and \cite{jossain} for the present article let us assume a stochastic basis
$ (\Om,\ftn,\P,(\ftn_r)_{r \in [0,T]}) $ as in Section \ref{section:preliminaries}.
\medskip

\underline{Relation to \cite{Geiss:wBMO}:}
In \cite{Geiss:wBMO} a class of weighted BMO spaces $\bmo_p^\Phi$ has been introduced.
For a positive c\`adl\`ag and adapted weight process $\Phi=(\Phi_t)_{t\in [0,T]}$ and
$p\in (0,\infty)$ we say that a continuous and adapted process  $A=(A_t)_{t\in [0,T]}$ with $A_0\equiv 0$  belongs to $\BMO_p^\Phi$ provided that
\begin{equation}\label{eqn:definition_BMO_p_Phi}
    \| A \|_{\BMO_p^\Phi} 
   := \sup \left \{
   \left \| \E \left ( \left | \frac{A_T - A_{\tau}}{\Phi_\tau} \right |^p \mitt \cF_\tau \right ) 
            \right \|_{L_\infty(\P)}^\frac{1}{p} \mitt
        \tau: \Omega \to [0,T] \mbox{ stopping time} 
            \right \}  < \infty.
\end{equation}
In the present article it is essential to use the concept of $\bmo_p^\Phi$ {\em locally in time}. To explain this let us 
look at Theorem \ref{theorem:main} where we have weight processes $(\weight_{p,s,u,t})_{u\in [s,t]}$ and 
$(\weight_{p,s,u,t}^{\xi,f})_{u\in [s,t]}$ for fixed $0\le s < t \le T$.  If we consider 
the special case $[s,t]=[0,T]$ and set
$\Phi_u := w_{p,0,u,T}\vee \vare$ 
and $\Phi'_u := w_{p,0,u,T}^{\xi,f}\vee \vare$ 
for any $\vare>0$ (the parameter $\vare>0$ is only formal, to get the weights strictly positive
to be in accordance with \cite{Geiss:wBMO}), then part of 
Theorem \ref{theorem:main} reads as
\begin{eqnarray}
        \| (Y_t-Y_0)_{t\in [0,T]}\|_{\BMO_p^{\Phi}} 
&\le& c_{\eqref{theorem:main}},\label{eqn:global_BMO_esimate_Y} \\
        \left \| \left (\int_0^t|Z_s|^2ds\right)_{t\in [0,T]}\right \|_{\BMO_{p/2}^{(\Phi')^2}} 
&\le& d^2_{\eqref{theorem:main}}. \label{eqn:global_BMO_esimate_Z} 
\end{eqnarray}
However, this "global" setting, i.e. $[s,t]=[0,T]$, would not give us estimates on the distribution 
of $Y_t-Y_s$ that take the size of $t-s$ into the account. Therefore 
Theorem \ref{theorem:main} provides local versions of \eqref{eqn:global_BMO_esimate_Y} and \eqref{eqn:global_BMO_esimate_Z}
in the following sense:
for an arbitrary sub-interval $[s,t]\subseteq [0,T]$ we show that, 
for any stopping time $\tau:\Om \to [s,t]$, 
\begin{eqnarray}
          \left (\E^{\ftn_\tau} |Y_t-Y_\tau|^p \right )^\frac{1}{p}
&\le & c_{\eqref{theorem:main}} \weight_{p,s,\tau,t},\\
 	      \left ( \E^{\cF_\tau}
 	       \left(\int_{\tau}^{t} |Z_r|^2 dr\right)^{\frac{p}{2}} \right )^\frac{1}{p}
&\le& d_{\eqref{theorem:main}} \weight_{p,s,\tau,t}^{\xi,f}.
\end{eqnarray}
The "main" weight process $(\weight_{p,s,u,t}^{\xi,f})_{u\in [s,t]}$ is obtained in
Assumption \ref{assumption:BMO}.
Because our approach requires a localization in $[s,t]$ the spaces $\bmo_p^\Phi$ could not be used in the form they 
have been defined in \cite{Geiss:wBMO} and we extracted the results from \cite{Geiss:wBMO} in a form needed  in Section \ref{sec:JN}.
This made it possible to perform Step 2, i.e. to deduce Theorem \ref{theorem:tail-estimate} from Theorem \ref{theorem:main}.

The Assumption \ref{assumption:BMO} measures the sensitivity of the initial data $(\xi,f)$ of our BSDE with respect 
to a class of conditional expectations in a natural way that might be interpreted as a property
related to directional Malliavin derivatives. But to prove Theorem \ref{theorem:main} we have to translate 
Assumption \ref{assumption:BMO} into the decoupling context and obtain the equivalent condition 
Assumption \ref{assumption:BMO_product}.
So Assumption \ref{assumption:BMO} and Theorem \ref{theorem:main} combine 
 \cite{jossain}  and \cite{Geiss:wBMO}: the weights originate from the decoupling techniques in \cite{jossain} and are
used in a context that localizes the BMO spaces from \cite{Geiss:wBMO}.
It should be mentioned that  Assumption \ref{assumption:BMO} might be seen also from the point of view
that we start with the initial data $(\xi,f)$ of the BSDE and {\it then} look for good or even best possible weight processes
$(\weight^\xi_{p,s,u,t})_{u\in [s,t]}$ and $(\weight^f_{p,s,u,t})_{u\in [s,t]}$.

Let us explain the importance of the localized approach, i.e. to consider subintervals $[s,t]\subseteq [0,T]$, 
by the example of decoupled Forward Backward SDEs (FBSDEs) treated in Section \ref{subsec:decoupled_FBSDEs}. There we consider
\begin{eqnarray*}
X_t &=& x + \int_0^t b(r,X_r)dr + \int_0^t \sigma(r,X_r)dW_r, \\
Y_t &=& g(X_T) + \int_t^T h(r,X_r,Y_r,Z_r)dr - \int_t^T Z_r dW_r, \label{eqn:JN-}
\end{eqnarray*}
for $t\in [0,T]$, where $x \in \R^d$ is fixed and the main assumption is that the functions $b,\sigma,g,h$ are uniformly 
Lipschitz in the state variables (see Assumption \ref{assumption:FBSDE} below).
A consequence of Theorem \ref{theorem:FBSDE_tail_Y} is, that there exists an absolute constant $c_0>0$
and constants $c,C>0$, depending on the parameters of the FBSDE, such that for any $0 \le s < t \le T$ we have
\begin{equation}\label{eqn:intro_Y-FBSDE_JN} 
       \P\left( \sup_{u \in [s,t]} \frac{|Y_u-Y_\tau|}{\sqrt{t-s}} > c \mu \nu \mitt \cF_s \right )
  \le  e^{1-\mu} + c_0 \P \left ( \sup_{u \in [s,t]} |X_u|^2 > \nu^2-1 \mitt \cF_s \right)
\end{equation}
for all $\mu,\nu > 0$. In the case that $\sigma$ is bounded, this improves to
\begin{equation} \label{eqn:intro_Y-FBSDE_b_JN} 
   \P\left( \sup_{u \in [s,t]} \frac{|Y_u-Y_\tau|}{\sqrt{t-s}} > c \mu \nu \mitt \cF_s \right )
  \le  e^{1-\mu} + c_0 \P \left ( \sup_{u \in [s,t]} |X_u|^2(t-u) > \nu^2-1 \mitt \cF_s \right). 
\end{equation}
Similar results are obtained for the process $Z$ of the FBSDE
in Theorem \ref{theorem:FBSDE_tail_Z}, and for the 
solution $(Y,Z)$ to the general non-Markovian BSDE in Theorem \ref{theorem:tail-estimate}.
The idea behind the inequalities  \eqref{eqn:intro_Y-FBSDE_JN} and \eqref{eqn:intro_Y-FBSDE_b_JN} is to 
minimize for a given $\la>0$ the right hand sides over all decompositions $\la=\mu\nu$. This is used
in Sections \ref{sec:spline_approximation} and \ref{sec:direct_simulation}.
Even though  \eqref{eqn:intro_Y-FBSDE_JN} and \eqref{eqn:intro_Y-FBSDE_b_JN} concern 
 a well-studied family of FBSDEs, the tail estimates we obtain in 
 \eqref{eqn:intro_Y-FBSDE_JN} and \eqref{eqn:intro_Y-FBSDE_b_JN}
(Theorems \ref{theorem:FBSDE_tail_Y} and \ref{theorem:FBSDE_tail_Z}) seem to be new.
Coming back to moment estimates there is another application
that shows the strength of the conditional approach. 
Let $s\in [0,T)$ and $n_s\ge 1$ such that $s+\frac{1}{n} \le T$ for $n\ge n_s$. Then Fatou's Lemma, 
the right-hand side continuity of the filtration, and Theorem \ref{theorem:main_FBSDE} for $p=2$ give 
\equa
      \liminf_{n\to \infty, n\ge n_s} \left ( n \int_s^{s+\frac{1}{n}} |Z_r|^2 dr \right )
& = & \E \left ( \liminf_{n\to \infty, n\ge n_s} n \int_s^{s+\frac{1}{n}} |Z_r|^2 dr \mitt \cF_s \right ) \\
&\le& \liminf_{n\to \infty, n\ge n_s} \E \left ( n \int_s^{s+\frac{1}{n}} |Z_r|^2 dr \mitt \cF_s \right ) \\
&\le& \begin{cases}
      C_{\eqref{theorem:main_FBSDE}}^2 [1+|X_s|^2] & \mbox{ under condition } (A_{b,\sigma}) \\
      D_{\eqref{theorem:main_FBSDE}}^2             & \mbox{ under conditions } (A_{b,\sigma}) \mbox{ and } (A_\sigma).
      \end{cases}
\tion

\pagebreak

\underline{Relation to \cite{jossain}:}
To prove our basic Theorem \ref{theorem:main} we use the decoupling technique from \cite{jossain}. This technique has to be extended into two directions:
\begin{itemize}\setlength\itemsep{-0.1em}
\item[--] 
      Similarly as the concept of the expected value is extended to the conditional expectation, some results from \cite{jossain} has to be extended 
      to a conditional context, see Section \ref{sec:conditional_results} and Proposition \ref{proposition:I_1}.
\item[--] 
      The above mentioned Assumption \ref{assumption:BMO} we need to translate into Assumption \ref{assumption:BMO_product} to apply the decoupling 
      technique. This translation is based on Proposition \ref{proposition:ugliness}. 
      Having in mind that every separable Banach space can be isometrically embedded into $C([0,1])$ by the Banach-Mazur Theorem and 
      that $\M=[0,1]$ is locally $\sigma$-compact, then Proposition \ref{proposition:ugliness} is also a statement about
      the conditional decoupling of random variables with values in separable Banach spaces. Therefore, Proposition \ref{proposition:ugliness} is
an infinite-dimensional version of Lemma \ref{lemma:conditional_equivalence}, where Lemma \ref{lemma:conditional_equivalence} is a conditional version of \cite[Lemma 4.20]{jossain}.

\end{itemize}

\medskip

The article is organized as follows. The main results are formulated in Section \ref{section:BSDE}. We also include proofs in Section \ref{section:BSDE}
as long as the decoupling technique from \cite{jossain} is not required. This technique is introduced in Section \ref{section:coupling}.
In Section \ref{sec:remaining_proofs_based_on_decoupling} we complete the proofs of the results in Section  \ref{section:BSDE}
with the methods from Section \ref{section:coupling}. Some applications of the estimates we obtained are illustrated in Section 
\ref{section:applications}. The Appendices A,B, and C contain some technical tools that were needed throughout the article.
}


\parindent=0pt
\section{Preliminaries}\label{section:preliminaries}

A constant with a subindex of the form \eqref{theorem:main_FBSDE} is a constant from the result that is numbered \ref{theorem:main_FBSDE}. 
For example, $c_{\eqref{theorem:main_FBSDE}},d_{\eqref{theorem:main_FBSDE}},C_{\eqref{theorem:main_FBSDE}}$ and $D_{\eqref{theorem:main_FBSDE}}$ 
are constants from Theorem \ref{theorem:main_FBSDE}.
We fix a finite number $T>0$ and work on the stochastic basis
$$ (\Om,\ftn,\P,(\ftn_r)_{r \in [0,T]}) $$
satisfying {\it {\st the} usual assumptions}. In particular, {\st $(\Omega,\cF,\P)$ is complete and in our case}
$\mathbb{F}:=(\ftn_r)_{r \in [0,T]}$ is the augmented filtration of a $d$-dimensional Brownian motion $W$,
$\ftn=\ftn_T$, and we assume that all paths of $W$ are continuous.
If we give a statement or a definition that involves a filtration, but the filtration is not mentioned explicitly, then $\mathbb{F}$ is used.
Moreover, the following notation will be used:
\begin{notation}\rm\hfill
\label{notation}
 \begin{enumerate}[(1)]
  \item
The Lebesgue-measure on $[0,T]$ is denoted by $\lambda$, and
	\begin{eqnarray*}
	 (\Om_0,\Sigma_0,\P_0) &:=& \left(\Om,\cF,\P\right), \\
	 (\Om_T,\Sigma_T,\P_T) &:=& \left([0,T]\times \Om,\mathcal{B}([0,T])\otimes\cF,\frac{\lambda}{T}\otimes\P\right).
	\end{eqnarray*}
  \item
Given a $\sigma$-algebra $\cG \subseteq \cF$ and $X \in L_1(\Om,\cF,\P)$, the conditional expectation of $X$ given $\cG$ is denoted by $\E^{\cG} X := \E \left[ X \mitt \cG \right]$.
  \item
For any $B \in \ftn$ of positive measure and any $A \in \ftn$ we let
	\begin{equation*}
	 \P_B(A):= \frac{\P(B\cap A)}{\P(B)}.
	\end{equation*}
  \item
For $0 \le s < t \le T$ we let
	\begin{equation}\label{eqn:sigma_algebras}
	 \cG_s^t := \sigma(W_r,r\le s) \vee \sigma(W_r-W_t, t < r \le T) \sptext{1}{and}{1}
	 \cH_s^t := \mathcal{B}([0,T]) \otimes \cG_s^t.
	\end{equation}
  \item
The (predictable) $\sigma$-algebra on $\Om_T$ generated by $(\cF_t)_{t \in [0,T]}$-adapted left-continuous processes is denoted by $\mathcal{P}$.
 \end{enumerate}
\end{notation}

In general, inequalities concerning random variables, for example $\E^{\cG} X \le c Y$, where $c>0$ is a constant, hold only almost surely. If it is obvious what measure is used,
we will just write $\E^{\cG} X \le c Y$.
{\st If $A$ is a subset of a metric space, then} we denote the interior of $A$ by ${\mathring A}$ and the closure of $A$ by $\overline{A}$.

\begin{definition}\label{definition:localcompact}
A complete metric space $\M\not = \emptyset$ is \emph{locally $\sigma$-compact}, if
there exist compact subsets \linebreak
$\emptyset \not = K_1 \subseteq K_2 \subseteq \dots $, such that $\overline{\mathring K}_n=K_n$ and
$\M=\bigcup_{n=1}^\infty \mathring{K}_n$.
\end{definition}
\begin{proposition}
A locally $\sigma$-compact $\M$ is separable. Moreover, if $(K_n)_{\st n \ge 1}$ are compact subsets as 
in Definition \ref{definition:localcompact}, and $\A \subseteq \M$ is a dense countable subset, then for 
any {\st $n \ge 1$} the set $\A_n := \A \cap K_n$ is dense in $K_n$.
\end{proposition}

\begin{definition}\label{definition:generator_space}
For $S \in \{0,T\}$ we use
 \begin{equation*}
  L_0(\Om_S,\Sigma_S,\P_S;C(\M))
 \end{equation*}
to denote the equivalence-classes\footnote{We identify $f$ and $g$ if $f(\eta,\cdot)=g(\eta,\cdot)$ for $\P_S$-a.e. $\eta \in \Om_S$.}
of $f:\Om_S\times \M \to \R$ that satisfy:
\begin{itemize}
 \item[(1)] $\eta \mapsto f(\eta,y)$ is $\Sigma_S$-measurable for all $y \in \M$,
 \item[(2)] $y \mapsto f(\eta,y)$ is continuous for all $\eta \in \Om_S$.
\end{itemize}
\end{definition}

We will need the Burkholder-Davis-Gundy-inequalities:

\begin{proposition}
[{\cite[p.160]{Revuz:Yor}, \cite[Proposition 4.2]{Barlow:Yor:82}}]
\label{proposition:BDG}
Let {\st $p \in(0,\infty)$}. Then there exists $\alpha_p,\beta_p>0$ such that for all (continuous) martingales $(M_t)_{t\in [0,T]}$ 
{\st with $M_0\equiv0$} we have:
 \begin{equation*}
      \alpha_p \left\|\langle M \rangle_t^\half\right\|_{L_p}
  \le \left\|\sup_{s \in [0,t]} |M_s| \right\|_{L_p}
  \le \beta_p \left\|\langle M \rangle_t^\half\right\|_{L_p}
 \end{equation*}
for all $t \in [0,T]$, {\st where $(\langle M \rangle_t)_{t \in [0,T]}$ is the quadratic variation process of $M$.
For $p\in [2,\infty)$ the constant $\beta_p>0$ can be chosen such that $\beta_p \le c \sqrt{p}$ for some some absolute $c>0$.}
\end{proposition}

Next we introduce  {\st the} {\em sliceable numbers.}

\begin{definition}[{\st cf. \cite[Definition 5.2]{jossain}}]\label{definition:BMO_S2}
Assume that $(c_r)_{r \in [0,T]}$ is predictable, $d$-dimensional and such that
 \begin{equation*}
   \| c \|_{\bmo(S_2)}
:= \sup_{t\in [0,T]} \left \| \E \left ( \int_t^T |c_s|^{2} ds | {\cF}_t \right ) \right \|_{L_\infty}^\half < \infty.
 \end{equation*}
Then we say $c \in \bmo(S_2)$. This is quantified using, for any $N \ge 1$,
$\sli_N(c) := \inf \vare$,
where the infimum is taken over all $\vare>0$ such that there are stopping times $0=\tau_0\le \tau_1 \le \cdots \le \tau_N = T$
with
\begin{equation*}  \sup_{k=1,...,N} \| \chi_{(\tau_{k-1},\tau_k]} c \|_{\bmo(S_2)} \le \varepsilon. \end{equation*}
Moreover, we let $\sli_\infty(c) := \lim_{N\to\infty} \sli_N(c)$.
\end{definition}
\smallskip
For our main application of sliceable numbers we introduce the function
 \begin{equation}\label{function:Phi}
  \Phi:(1,\infty) \to (0,\infty), \quad \Phi(q) = \left(1+\frac{1}{q^2}\log\left(1+\frac{1}{2q-2}\right)\right)^{1/2} -1,
 \end{equation}
so that $\Phi$ is continuous and decreasing, with $\lim_{q\to \infty} \Phi(q) = 0$ and 
$\lim_{q\to 1     } \Phi(q) = \infty$. Furthermore, we let
 \equa
  \Psi &:& \Big \{(\gamma,q)\in [0,\infty)\times (1,\infty): 0\le \gamma < \Phi(q)<\infty  \Big \}\to [0,\infty), \\
  \Psi(\gamma,q) &:= & \left ( \frac{2}{1-\frac{2q-2}{2q-1} e^{q^2 [\gamma^2 + 2\gamma]}}\right )^\frac{1}{q}.
 \tion
{\st The concept of sliceable numbers is motivated by Proposition \ref{proposition:scliceable_rh} below. To formulate 
this statement we need the following definition:}

\begin{definition}
\label{definition:A_p}
Let $M=(M_t)_{t\in [0,T]}$ be a martingale with $M_0\equiv 0$ such that
 $ \mathcal{E}(M) = (\mathcal{E}(M)_t)_{t \in [0,T]} := (e^{M_t-\frac{1}{2}\langle M \rangle_t})_{t \in [0,T]}$
is a martingale as well, and let $q\in (1,\infty)$. If
  \begin{equation*}
    \rh_q(\mathcal{E}(M))^q := \sup_{\tau} \left\| \E^{\cF_\tau} \left| \frac{\mathcal{E}(M)_T}{\mathcal{E}(M)_\tau}\right|^q \right\|_\infty < \infty,
  \end{equation*}
where the supremum is taken over all stopping times {\st $\tau:\Omega\to [0,T]$},
we say\footnote{$\rh$ stands for Reverse H\"older.} $\mathcal{E}(M) \in \rh_q$.
\end{definition}

\begin{proposition}[{\st \cite[Theorem 5.25]{jossain}\label{proposition:scliceable_rh}}]
Assume that $c \in \bmo(S_2)$ is $d$-dimensional, and that for some {\st $N \ge 1$} it holds {\st that} $\sli_{N}(c) < \Phi(q)$.
Then, putting $(M_t)_{t \in [0,T]} := (\int_0^t c_r dW_r)_{t \in [0,T]}$, we have
 \begin{equation*}
  \rh_q (\mathcal{E}(M)) \le \big [\Psi(\sli_{N}(c),q)\big ]^N.
 \end{equation*}
In particular, if $M$ is {\em sliceable}, i.e.
$\sli_{\infty}(c) = 0$,
then for all $q>1$ there exists an {\st $N \ge 1$} such that \linebreak
$\sli_N(c) < \Phi(q)$, so that
$\mathcal{E}(M) \in \bigcap_{q \in (1,\infty)} \rh_q$.
\end{proposition}
\medskip

We end with an extension of Fefferman's inequality, which was proven in 
{\st \cite[Corollary 5.19]{jossain}}
(see also \cite[Lemma 1.6]{Delbaen-Tang} and \cite[Theorem 1.1(iii)]{Banuelos-Bennett}). Note that here both $X$ and $Y$ may be multidimensional.

\begin{proposition}
\label{proposition:Fefferman}
Assume that $X \in \bmo(S_2)$ and that $Y= (Y_r)_{r \in [0,T]}$ is predictable and such that
 \begin{equation*}
  \|Y\|_{H_p(S_2)}^p:=\E\left(\int_0^T |Y_r|^2 dr \right)^{p/2} < \infty
 \end{equation*}
for some $p\in [1,\infty)$. Then
$\left\| \int_0^T |X_r| |Y_r| d_r \right\|_{L_p} \le \sqrt{2p} \|Y\|_{H_p(S_2)} \|X\|_{\bmo(S_2)}$.
\end{proposition}

In this article we deduce conditional estimates on subintervals $[s,t] \subseteq [0,T]$,
and for this we need the following conditional version of Proposition \ref{proposition:Fefferman}:

\begin{proposition}\label{proposition:Fefferman_cond_new}
Assume that $X \in \bmo(S_2)$ and that $Y= (Y_r)_{r \in [0,T]}$ is predictable and such that \linebreak
{\st $\|Y\|_{H_p(S_2)} < \infty$}
for some $p\in [1,\infty)$, and let 
$c_p=(\sqrt{2p})^p$. Then we have for all $0 \le s < t \le T$ that
 \begin{equation*}
    \E^{\cF_s} \left(\int_s^t |X_r| |Y_r| dr\right)^p
\le c_p \left(\E^{\cF_s} \left(\int_s^t |Y_r|^2 dr \right)^{\frac{p}{2}}\right)
        \sup_{r \in [s,t]} \left\| \E^{\cF_r} \int_r^t |X_u|^2 du \right\|_\infty^{\frac{p}{2}}.
 \end{equation*}
\end{proposition}


\section{Weighted BMO-estimates for BSDEs}
\label{section:BSDE}

{\st First we present our results in the general non-Markovian context in Section \ref{subsec:Non-Markovian_BSDEs}.
Then the results are illustrated for decoupled FBSDEs in Section \ref{subsec:decoupled_FBSDEs}
where we also discuss their sharpness in Examples \ref{example_bounded_below2} and \ref{example:bounded_below} .}

\subsection{Non-Markovian BSDEs}
\label{subsec:Non-Markovian_BSDEs}

We consider BSDEs of type
\begin{equation}\label{equation:BSDE1}
   Y_t = \xi + \int_t^T f(s,Y_s,Z_s)ds - \int_t^T Z_s dW_s, \qquad t \in [0,T],
\end{equation}
where $\xi$ is $\cF_T$-measurable, and
$f \in \mathcal{L}_0(\Om_T,\mathcal{P},\P_T;C(\R^{d+1}))$\footnote{This means that $\eta \mapsto f(\eta,x)$ is $\mathcal{P}$-measurable for all $x \in \R^{d+1}$.}.
Our strategy is to assume that $(Y,Z)$ is a solution of \eqref{equation:BSDE1}, and assume some further conditions on $Z$ 
in order to get an $L_p$-solution for {\st $p\in [2,\infty)$}.
In Example \ref{example:conditions_satisfied} we present some cases when these conditions are satisfied.
For $p\in [2,\infty)$ and $\theta \in [0,1]$, we consider the conditions:

\begin{enumerate}[{(C1)}]
\item
      There are
      $L_y,L_z\ge 0$ such that for all $(t,\om,y_0,y_1,z_0,z_1)$ one has
      \begin{equation*}
      |f(t,\om,y_0,z_0)-f(t,\om,y_1,z_1)|
      \le L_y |y_0-y_1| + L_z [1+|z_0|+|z_1|]^\theta |z_0-z_1|.
      \end{equation*}
\item $\int_0^T |f(s,0,0)| ds \in \cL_p$.
\item $\left ( \int_0^T |Z_s|^2 ds \right )^\frac{1}{2} \in \cL_p$.
\item[{(\rm C3')}] $  \int_0^T |Z_s|^{1+\theta} ds  \in \cL_p$.
\end{enumerate}

Assumptions (C1) and (C2) are conditions on the data of the BSDE, implicit conditions on the $Z$-process are (C3) and (C3').
\begin{lemma}[{\cite[Lemma 6.2]{jossain}}]\label{lemma:Ynice}
Assume that {\rm (C1)}-{\rm (C3)} and {\rm (C3')} hold for some $p \in [2,\infty)$ and $\theta \in [0,1]$. Then
 \begin{equation*}
  \int_0^T |f(s,Y_s,Z_s)|ds + \sup_{t \in [0,T]} |Y_t| \in \mathcal{L}_p.
 \end{equation*}
\end{lemma}

\medskip

Another implicit condition is the following "fractional BMO-assumption":
\begin{enumerate}
\item [{\rm (C4)}] We assume that
 \begin{equation*}
   \| |Z|^\theta \|_{\bmo(S_2)}
= \sup_{t\in [0,T]} \left \| \E \left ( \int_t^T |Z_s|^{2\theta} ds | {\cF}_t \right ) \right \|_{\infty}^\half < \infty,
 \end{equation*}
and fix a non-increasing sequence $s=(s_N)_{N\ge 1} \subseteq [0,\infty)$ such that
 \begin{equation*} \sli_N(|Z|^\theta) \le s_N, \end{equation*}
and put $s_\infty := \lim_{N\to\infty} s_N$. If $s_\infty=0$, {\st then} we let $p_{{\rm (C4)}}=\frac{3}{2}$, and if $s_\infty>0$,
{\st then} we let
		  \begin{equation*} p_{{\rm (C4)}} := \frac{\Phi^{-1}(2\sqrt{2}L_z s_\infty)}{\Phi^{-1}(2\sqrt{2}L_z s_\infty)-1}, \end{equation*}
where the function $\Phi$ is defined in \eqref{function:Phi}.
\end{enumerate}

\smallskip

First we show that using (C4) we may drop the assumption (C3'):

\begin{lemma}\label{lemma:drop_condition}
For all $p \in [2,\infty)$ we have the following relations: 
\begin{enumerate}[{\rm (i)}]
\item If $\theta=0$, then {\rm (C4)} holds, and {\rm (C3)} $\Rightarrow$ {\rm (C3')}.
\item If $\theta=1$, then {\rm (C4)} $\Rightarrow$ {\rm (C3')} $\Rightarrow$ {\rm (C3)}.
\item If $\theta \in (0,1)$ and {\rm (C4)} holds, then {\rm (C3)} $\Rightarrow$ {\rm (C3')}.
\end{enumerate}
\end{lemma}
\begin{proof}
(i) is obvious and (ii) follows immediately from John-Nirenberg inequality {\cite[Theorem 2.1]{Kazamaki}}.
Proposition \ref{proposition:Fefferman} applied to $X = |Z|^\theta$ and $Y = |Z|$ implies (iii).
\end{proof}

\begin{remark}\rm
In addition to Lemma \ref{lemma:drop_condition}, the condition {\rm (C4)} has an even more important role that we describe now.
In our results, conditions {\rm (C4)} and {\rm (C1)} are assumed to hold for the same $\theta \in [0,1]$. Then,
applying Proposition \ref{proposition:scliceable_rh}, we have that a certain martingale satisfies the reverse Hölder inequality.
This martingale is used to handle the quadratic or sub-quadratic nature of the generator $f$ in the $z$-variable.
If the number $s_\infty$ in {\rm (C4)} equals zero, then the reverse Hölder inequalities are satisfied for all indices $q \in (1,\infty)$.
On the other hand, if $s_\infty>0$, then there exists $q_0 \in (1,\infty)$ such that the reverse Hölder inequalities are satisfied for all $q \in (1,q_0)$.
From this it follows that in the case $s_\infty>0$ we need to assume more integrability than in the case $s_\infty=0$, and this is the reason for introducing the constant $p_{{\rm (C4)}}$.
Note that in the uniformly Lipschitz case, i.e. $\theta=0$, the condition {\rm (C4)} is satisfied and $s_\infty=0$.
In the sub-quadratic case, i.e. $\theta \in (0,1)$, a sufficient condition for $s_\infty=0$ is, that there exists an $\eta \in (\theta,1]$ such that $\| |Z|^\eta \|_{\rm{BMO}(S_2)} < \infty$ \linebreak
(see {\st \cite[Remark 6.4]{jossain}}).
\end{remark}

\begin{example}\label{example:conditions_satisfied}\rm
$ $
\begin{enumerate}[(i)]
 \item Assume that $f$ satisfies {\rm (C1)} and {\rm (C2)} with $\theta=0$ and $p > 1$, and that $\xi \in L_p$.
       Then there exists a unique solution $(Y,Z)$ of \eqref{equation:BSDE1}, and {\rm (C3)-(C4)} are satisfied with $\theta=0$.
       This follows for example from {\cite[Theorem 4.2]{Briand:02}}. Note that since $\theta=0$, we have $s_\infty=0$.
 \item Assume that $f$ satisfies {\rm (C1)} and {\rm (C2)} with $\theta=1$ and $p = \infty$, and that $\xi \in L_\infty$.
       Then there exists a solution $(Y,Z)$ of \eqref{equation:BSDE1}
       such that {\rm (C3)-(C4)} are satisfied with $\theta=1$ and all $p \in [2,\infty)$.
       This follows for example from {\cite[Theorem 2.6 and Lemma 3.1]{Morlais}} (see also \cite{Briand-Hu}).
\item Assume that $f$ satisfies {\rm (C1)} with $\theta_0 \in (0,1)$, and is such that 
$\sup_{(r,\om)}|f(r,\om,0,0)| < \infty$. Also, assume that $\xi \in \rm{cExp}$, which means that there exists some $\mu \in (0,\infty)$ such that
 \begin{equation*}
  \sup_{t \in [0,T)} (T-t)\left\| \E \left[ e^{\mu |\xi|} \mitt {\cF_t} \right] \right\|_\infty < \infty.
 \end{equation*}
       Then there exists a solution $(Y,Z)$ of \eqref{equation:BSDE1} such that {\rm (C3)-(C4)} are satisfied with $p =2$ and all $\theta \in (0,1)$, so 
       that $s_\infty=0$ (see \cite[Theorem 6.13]{jossain}).
\end{enumerate}
\end{example}

Our final assumption is a weighted $\bmo$-condition on $\xi$ and $f$ on a subinterval $[s,t]\subseteq[0,T]$. This is used in the following way: if (C1)-(C4) are satisfied
and Assumption \ref{assumption:BMO} holds on an interval $[s,t]$,
then {\em on this interval} we have a weighted BMO-estimate and a tail estimate of $(Y,Z)$.

\begin{assumption}\label{assumption:BMO}
{\st Let $p\in [2,\infty)$ and $0 \le s < t \le T$.} There are non-negative c\`adl\`ag {\st processes}
$(\weight^\xi_{p,s,u,t})_{u \in [s,t]}$ and $(\weight^f_{p,s,u,t})_{u \in [s,t]}$
{\st such that $((\weight^\xi_{p,s,u,t})^p)_{u \in [s,t]}$ and $((\weight^f_{p,s,u,t})^p)_{u \in [s,t]}$
are supermartingales and}
which satisfy{\st,} for any 
$u \in [s,t]$,
\begin{enumerate}
\item [{\rm (C5)}] ${\st \left ( \E^{\cF_u}|\xi-\E^{\cG_u^t}\xi|^p \right )^\frac{1}{p}}  \le \weight^\xi_{p,s,u,t}$,
\item [{\rm (C6)}] ${\st \left ( \E^{\cF_u} \left(\int_u^T \sup_{y,z} |f(r,y,z)-(\E^{\cH_u^t}f)(r,y,z)|dr\right)^p \right )^\frac{1}{p}} \le \weight^f_{p,s,u,t}$,
\end{enumerate}
where {\st $\cG_u^t$ and $\cH_u^t$ are given in \eqref{eqn:sigma_algebras} and}
$(\E^{\cH_u^t}f):\Om_T \to C(\R\times\R^d)$ is the\footnote{Existence and uniqueness of such a process is proven in Lemma \ref{lemma:(ii)} below.}
$\cH_u^t$-measurable process with
 \begin{equation*}
  \P_T\left( \E^{\cH_u^t} (f(x)) = (\E^{\cH_u^t} f)(x)\right)=1 \mbox{ for all } x \in \R\times\R^d.
 \end{equation*}
{\st To shorten the notation, we use
\[ \weight_{p,s,u,t}^{\xi,f} :=  \left ( (\weight_{p,s,u,t}^\xi)^p + (\weight_{p,s,u,t}^f)^p \right )^\frac{1}{p}.\]
 }
 \end{assumption}

\begin{remark}\rm
 For a fixed $u \in (s,t]$, the weight $\weight^\xi_{p,s,u,t}$ is an upper bound for 
{\st $\left ( \E^{\cF_u}|\xi-\E^{\cG_u^t}\xi|^p\right )^\frac{1}{p}$}, so we expect $\weight^\xi_{p,s,u,t}$ to depend on $u$ and $t$, but not on $s$.
We use a notation where the $s$ is included, since we want to emphasize the fact that Assumption \ref{assumption:BMO} is an assumption on the behaviour of $(\xi,f)$ \emph{on the interval} $[s,t]$.
\end{remark}

We are ready to give our main result.

\begin{theorem}
\label{theorem:main}
Assume {\rm (C1)-(C6)} for $\theta \in [0,1]$, $p \in [2,\infty)\cap(p_{{\rm (C4)}},\infty)$, and $0 \le s < t \le T$.
Then the following assertions hold true:
 \begin{enumerate}
 	\item[$\mathbf{(i)}$] There exists
 	          $c_{\eqref{theorem:main}}>0$ depending at most on
 	          $(T,d,p,L_y,L_z,(s_N)_{N\in \N})$ 
 	          such that for any stopping time $\tau:\Om \to [s,t]$ we have
               \begin{equation}\label{equation:BMO}
 	            {\st \left (\E^{\ftn_\tau} |Y_t-Y_\tau|^p \right )^\frac{1}{p}
 	            \le c_{\eqref{theorem:main}} \weight_{p,s,\tau,t}},
 	           \end{equation}
 	          where
               \equa
 	            \weight_{p,s,u,t}^p = {\st \left(\weight^{\xi,f}_{p,s,u,t} \right )^{p}}
 	           +  \E^{\ftn_u}  \left(\int_u^t|f(r,0,0)|dr\right)^p 
 	            +     (t-u)^p\left[ \E^{\ftn_u} \left( |\xi| +
 	            \int_t^T|f(r,0,0)|dr\right)^p \right].
 	           \tion
 	\item[$\mathbf{(ii)}$] There exists
 	          $d_{\eqref{theorem:main}}>0$ depending at most on $(T,d,p,L_y,L_z,(s_N)_{N\in \N})$
 	          	          such that for any stopping time $\tau:\Om \to [s,t]$ we have
 	           \begin{equation*}
 	            {\st \left ( \E^{\cF_\tau}
 	            \left(\int_{\tau}^{t} |Z_r|^2 dr\right)^{\frac{p}{2}} \right )^\frac{1}{p}
 	            \le d_{\eqref{theorem:main}}
 	            \weight_{p,s,\tau,t}^{\xi,f}.}
 	           \end{equation*}
 	         \end{enumerate}
\end{theorem}
{\st Theorem \ref{theorem:main}} is proved in Section \ref{subsection:theorem:main} below.
The main {\st application} of Theorem \ref{theorem:main} are the following tail estimates:

\begin{theorem}\label{theorem:tail-estimate}
Under the assumptions of Theorem \ref{theorem:main} there exists an absolute constant ${\st c_0}>0$
such that for any stopping time $\tau:\Om \to [s,t]$ we have
\begin{align*}
&\mathbf{(i)} \
  \P_B\left( \sup_{u \in [\tau,t]} \frac{|Y_u-Y_\tau|}{c_{\eqref{theorem:main}}}  > \lambda + {\st c_0} \mu \nu \right)
	\le \hfill e^{1-\mu} \P_B\left( \sup_{u \in [\tau,t]}\frac{|Y_{u}-Y_{\tau}|}{c_{\eqref{theorem:main}}} > \lambda \right)
	+ {\st c_0} \P_B\left(\sup_{u \in [\tau,t]}\weight_{p,s,u,t} > \nu \right), \\
&\mathbf{(ii)} \
\P_B \left( \sup_{u \in [\tau,t]}\left|\frac{\int_{\tau}^{u} Z_r dW_r}{d_{\eqref{theorem:main}} \beta_p}\right| > \lambda + {\st c_0} \mu\nu \right)
\le \hfill e^{1-\mu} \P_B\left( \sup_{u \in [\tau,t]}\left|\frac{\int_{\tau}^{u} Z_r dW_r}{d_{\eqref{theorem:main}} \beta_p}\right| > \lambda \right)
+ {\st c_0} \P_B\left(\sup_{u \in [\tau,t]} {\st \weight_{p,s,u,t}^{\xi,f}}> \nu \right), \\
&\mathbf{(iii)} \
\P_B \left(
\frac{\left(\int_{\tau}^{t} |Z_r|^2 dr\right)^{\frac{1}{2}}}{d_{\eqref{theorem:main}}}
> \lambda + {\st c_0} \mu\nu \right)
\le \hfill e^{1-\mu} \P_B \left(
\frac{\left(\int_{\tau}^{t} |Z_r|^2 dr\right)^{\frac{1}{2}}}{d_{\eqref{theorem:main}}} > \lambda \right)
+  {\st c_0} \P_B\left(\sup_{u \in [\tau,t]} {\st \weight_{p,s,u,t}^{\xi,f}}
> \nu \right),
\end{align*}
for all $\lambda,\mu,\nu>0$ and any $B \in \cF_\tau$ of positive measure, and where $\beta_p$ is the constant from Proposition \ref{proposition:BDG}.
\end{theorem}

\begin{proof}
{\st As the tail estimates follow from Theorem \ref{theorem:Stefan1} below in Appendix C},
we show that the assumptions of Theorem \ref{theorem:Stefan1} follow from Theorem \ref{theorem:main}.
Assume (C1)-(C6) for $\theta \in [0,1]$, 
$p \in [2,\infty)\cap(p_{{\rm (C4)}},\infty)$, and $0 \le s < t \le T$.
Let $\epsilon>0$,
$\thteta \in (0,\half)$, and $R:= t-s$. \\

$\mathbf{(i)}$ Define, for $r \in [0,R]$,
 \equa
  \gtn_r	&:=& \ftn_{r+s}, \\
  A_r		&:=& \frac{(Y_{r+s}-Y_s) \thteta^{1/p}}{c_{\eqref{theorem:main}}}, \\
  \Psi_{r}	&:=& \weight_{p,s,r+s,t}\vee\epsilon,
 \tion
where {\st $(\weight_{p,s,u,t})_{u\in [s,t]}$} is the weight process from Theorem \ref{theorem:main}. For $0 \le a < b$ and a filtration $(\cH_r)_{r \in [a,b]}$ we introduce the notation
\begin{equation*}
  \mathcal{S}^\cH_{a,b} := \left\{ \tau:\Om\to[a,b] \mitt \tau \text{ is a } (\cH_r)_{r \in [a,b]} \text{-stopping time} \right\},
\end{equation*}
so that in particular $\mathcal{S}^{\cG}_{0,R}+s = \mathcal{S}^\cF_{s,t}$.
Then the assumptions of Theorem \ref{theorem:Stefan1} are fulfilled.
As the other assumptions are obvious, we will only show that equation \eqref{equation:(1)} holds.
Using Theorem \ref{theorem:main}
we deduce
 \equa
 	    \sup_{\tau \in \mathcal{S}^{\cG}_{0,R}}
		\left|\left|\E \left[ \frac{|A_{R}-A_{\tau}|^p}{\Psi_{\tau}^p} \mitt \cG_\tau \right]\right|\right|_\infty
  &=& \sup_{\tau \in \mathcal{S}^{\cG}_{0,R}}
		\left|\left|\E \left[ \frac{|Y_t-Y_{\tau+s}|^p}{\weight_{p,s,\tau+s,t}^p\vee\epsilon^p} \mitt \ftn_{\tau + s} \right]\right|\right|_\infty
		\frac{\thteta}{c_{\eqref{theorem:main}}^p} \\
	&=& \sup_{\tilde\tau \in \mathcal{S}_{s,t}^{\cF}}
		\left|\left|\E \left[ \frac{|Y_t-Y_{\tilde\tau}|^p}{\weight_{p,s,\tilde\tau,t}^p\vee\epsilon^p} \mitt \ftn_{\tilde\tau} \right]\right|\right|_\infty
		\frac{\thteta}{c_{\eqref{theorem:main}}^p}
  \le \thteta.
 \tion
Hence we have by Chebyshev's inequality that for any 
$\nu>0$, $\tau \in \mathcal{S}^\cG_{0,R}$, and $B \in \gtn_\tau$ of positive measure:
 \begin{equation*}
  \P_B(|A_{R}-A_\tau|> \nu)
  \le \P_B(|A_{R}-A_\tau| > \Psi_{\tau})
    + \P_B(\Psi_{\tau} > \nu)
  \le \thteta + \P_B(\Psi_{\tau}>\nu).
 \end{equation*}
Letting $\epsilon \to 0$ implies the claim. \\

$\mathbf{(ii)}$ The claim follows analogously as $\mathbf{(i)}$, when we choose
 \equa
  \gtn_r	&:=& \ftn_{r+s}, \\
  A_r		&:=& \frac{\int_s^{r+s} Z_v dW_v
  	        \thteta^{1/p}}{d_{\eqref{theorem:main}} \beta_p}, \\
  \Psi_{r}	&:=& {\st \weight_{p,s,r+s,t}^{\xi,f} \vee\epsilon},
 \tion
where $\beta_p$ is the constant from Proposition \ref{proposition:BDG},
as then we have by Theorem \ref{theorem:main} that
 \equa
       \sup_{\tau \in \mathcal{S}^{\cG}_{0,R}}
   \left|\left|\E \left[ \frac{|A_{R}-A_{\tau}|^{p}}{\Psi_{\tau}^{p}} \mitt \cG_\tau \right]\right|\right|_\infty
 &=& \sup_{\tau \in \mathcal{S}^{\cG}_{0,R}}
 \left|\left|\E \left[ \frac{\left|\int_{\tau+s}^{t} Z_v dW_v\right|^{p}}{\st \left(\weight_{p,s,\tau+s,t}^{\xi,f}
 	 \right)^p \vee\epsilon^p} \mitt \ftn_{\tau + s} \right]\right|\right|_\infty
 \frac{\thteta}{\left(d_{\eqref{theorem:main}} \beta_p\right)^p} \\
 &\le& \sup_{\tilde\tau \in \mathcal{S}_{s,t}^{\cF}}
     \left|\left|\E \left[ \frac{(\int_{\tilde\tau}^{t} |Z_v|^2 dv)^{\frac{p}{2}}}{\st \left(\weight_{p,s,\tilde\tau,t}^{\xi,f}
 	  \right)^p \vee\epsilon^p} \mitt \cF_{\tilde\tau} \right]\right|\right|_\infty
 \frac{\thteta}{\left(d_{\eqref{theorem:main}}\right)^p}
 \le \thteta.
 \tion

\bigskip

$\mathbf{(iii)}$ The claim follows analogously, when we choose
\equa
\gtn_r	&:=& \ftn_{r+s}, \\
A_r		&:=& \frac{(\int_s^{r+s} |Z_v|^2 dv)^{\frac{1}{2}}\thteta^{\frac{1}{p}}}{d_{\eqref{theorem:main}}}, \\
\Psi_{r}	&:=& {\st \weight_{p,s,r+s,t}^{\xi,f} \vee\epsilon}.
\tion
\end{proof}

\subsection{Decoupled FBSDEs}\label{subsection:FBSDE}
\label{subsec:decoupled_FBSDEs}

We fix $x \in \R^d$ and consider the decoupled FBSDE
 \begin{eqnarray}\label{equation:BSDE}
   X_t &=& x + \int_0^t b(r,X_r)dr + \int_0^t \sigma(r,X_r)dW_r,     \quad t \in [0,T], \nonumber \\
   Y_t &=& g(X_T) + \int_t^T h(r,X_r,Y_r,Z_r)dr - \int_t^T Z_r dW_r, \quad t \in [0,T].
 \end{eqnarray}

\begin{assumption}\label{assumption:FBSDE}
The functions $b:[0,T]\times\R^d \to \R^d$, $\sigma:[0,T]\times\R^d\to \R^{d\times d}$ and
$h:[0,T]\times \R^d\times \R \times \R^{d} \to \R$ are continuous, and furthermore we assume:

\begin{itemize}
 \item[$(A_{b,\sigma})$] There exists $L_{b,\sigma}>0$ such that for all $0 \le t \le T$ and $x,y\in\R^{d}$ one has
		\begin{equation*}
		  |b(t,x)-b(t,y)| + |\sigma(t,x)-\sigma(t,y)| \le L_{b,\sigma} |x-y|.
		\end{equation*}
 \item[$(A_g)$]  There exists $L_g>0$ such that for all $x,y \in \R^d$ one has
                 \begin{equation*}
                   |g(x)-g(y)| \le L_g |x-y|.
                 \end{equation*}
 \item[$(A_h)$]  There exists $L_h>0$ such that for all $0 \le t \le T$ and $x_i,z_i \in \R^d$, $y_i \in \R$, $i=1,2$, one has
                 \begin{equation*}
                   |h(t,x_1,y_1,z_1)-h(t,x_2,y_2,z_2)| \le L_h (|x_1-x_2| + |y_1-y_2| + |z_1-z_2|).
                 \end{equation*}
\end{itemize}
\end{assumption}

\begin{remark}\rm\hfill
 \begin{itemize}
   \item[(1)] In particular it follows from Assumption \ref{assumption:FBSDE}, that there exist constants $L_h,K_h,K_{b,\sigma}>0$ such that we have
    \equa
     |h(t,x,y,z)|              &\le& K_h + L_h (|x| + |y| + |z|), \\
     |b(t,x)| + |\sigma(t,x)|  &\le&K_{b,\sigma}(1+|x|),
    \tion
   for all $(t,x,y,z) \in [0,T]\times \R^d\times \R\times \R^{d}$.
   \item[(2)] Under Assumption \ref{assumption:FBSDE}, there exists a  unique solution $(X,Y,Z)$ to FBSDE \eqref{equation:BSDE} and it holds
    \begin{equation*}
     \E \left[ \sup_{r \in [0,T]} |X_r|^p + \sup_{r \in [0,T]} |Y_r|^p + \left(\int_0^T |Z_r|^2 dr \right)^{\frac{p}{2}} \right] < \infty
    \end{equation*}
   for all $p \ge 2$ (see for example {\cite[Theorem 4.2]{Briand:02}}).
\end{itemize}
\end{remark}

Assumption $(A_{b,\sigma})$ is a classical assumption for the forward equation. If $(A_{b,\sigma})$ holds, then we have a weighted BMO-estimate for the forward process $X$ (see Lemma \ref{lemma:BMO_X}). Using this together with $(A_g)$ and $(A_h)$ we receive a weighted BMO-estimate for $(Y,Z)$, which gives us a tail-estimate for $(Y,Z)$.
If we assume in addition to $(A_{b,\sigma})$ that $\sigma$ is bounded, then the BMO-estimates for $(X,Y,Z)$ are improved.

 \begin{itemize}
  \item[($A_\sigma$)] {\em There exists $K_\sigma>0$ such that for all $0 \le t \le T$ and $x \in \R^d$ one has
                       \begin{equation*}
                        |\sigma(t,x)| \le K_\sigma.
                       \end{equation*}}
 \end{itemize}

{\st First let us give the weights from Assumption \ref{assumption:BMO} for the FBSDE case:}

\begin{example}\label{example:special_case}
Assume that Assumption \ref{assumption:FBSDE} holds. Then
assumptions {\rm (C1)-(C6)} hold true for $\theta=0$, all \linebreak
$p\in [2,\infty)$, and all $0 \le s < t \le T$.
Moreover, there exists $c_{\eqref{example:special_case}}>0$ depending at most on \linebreak
$(T,d,p,L_g,L_h,L_{b,\sigma},K_{b,\sigma})$ such that we may choose
\begin{equation*} 
\weight^f_{p,s,u,t}=\weight^\xi_{p,s,u,t}={\st c_{\eqref{example:special_case}} (t-u)^{1/2} \left(1+\E \left[ \sup_{r \in [u,t]}|X_r|^p \mitt \cF_u \right]\right)^\frac{1}{p}} 
\end{equation*}
for all $0 \le s \le u \le t \le T$. If additionally $(A_\sigma)$ holds, then there exists $d_{\eqref{example:special_case}}>0$ depending at most on $(T,d,p,L_g,L_h,L_{b,\sigma},K_{\sigma})$ such that we may choose
 \begin{equation*} \weight^f_{p,s,u,t}=\weight^\xi_{p,s,u,t}={\st d_{\eqref{example:special_case}} (t-u)^{1/2}}. \end{equation*}
\end{example}

{\st Example \ref{example:special_case} is proved in Section \ref{subsec:proof_lemma:special_case} below.
Now our first result is a consequence of Theorem \ref{theorem:main}:}
\bigskip

\begin{theorem}\label{theorem:main_FBSDE}
Assume that Assumption \ref{assumption:FBSDE} holds 
and let {\st $p\in [2,\infty)$}. Then the following assertions hold true:
\begin{itemize}
 \item[$\mathbf{(i)}_Y$] There exists
            $c_{\eqref{theorem:main_FBSDE}}>0$, depending at most on $(T,d,p,L_h,L_g,L_{b,\sigma},K_{b,\sigma},K_h)$, such that for any \linebreak
            $0 \le s < t \le T$ and any stopping time $\tau:\Om \to [s,t]$ we have
             \begin{equation*}
                  \E \left(|Y_t-Y_\tau|^p\mitt \ftn_\tau \right)
           	  \le c_{\eqref{theorem:main_FBSDE}}^p (t-\tau)^{p/2} {\st [1 + |X_\tau|^p]}.
             \end{equation*}
 \item[$\mathbf{(i)}_Z$] There exists
            $C_{\eqref{theorem:main_FBSDE}}>0$ depending at most on $(T,d,p,L_h,L_g,L_{b,\sigma},K_{b,\sigma})$
            such that for any $0 \le s < t \le T$ and any stopping time $\tau:\Om \to [s,t]$ we have
             \begin{equation*}
               \E \left( \left(\int_{\tau}^{t} |Z_r|^2 dr\right)^{\frac{p}{2}}\mitt \cF_\tau \right)
             \le C_{\eqref{theorem:main_FBSDE}}^p
                 (t-\tau)^{\frac{p}{2}}{\st [1+|X_\tau|^p]}.
             \end{equation*}
\end{itemize}
If additionally $(A_\sigma)$ holds, then we have:
\begin{itemize}
 \item[$\mathbf{(ii)}_Y$] There exists
            $d_{\eqref{theorem:main_FBSDE}}>0$, depending at most on $(T,d,p,L_h,L_g,L_{b,\sigma},K_{b,\sigma},K_h,K_{\sigma})$, such that for any $0 \le s < t \le T$ and any stopping time $\tau:\Om \to [s,t]$ we have
             \begin{equation*}
               \E \left( |Y_t-Y_\tau|^p \mitt \ftn_\tau \right) \le d_{\eqref{theorem:main_FBSDE}}^p(t-\tau)^{p/2}
                    {\st [1+|X_\tau|^p(t-\tau)^{p/2}]}.
             \end{equation*}
 \item[$\mathbf{(ii)}_Z$] There exists
        $D_{\eqref{theorem:main_FBSDE}}>0$ depending at most on $(T,d,p,L_h,L_g,L_{b,\sigma},K_{\sigma})$
            such that for any $0 \le s < t \le T$ and any stopping time $\tau:\Om \to [s,t]$ we have
             \begin{equation*}
               \E^{\cF_\tau} \left(\int_{\tau}^{t} |Z_r|^2 dr\right)^{\frac{p}{2}}
             \le D_{\eqref{theorem:main_FBSDE}}^p 
                 (t-\tau)^{\frac{p}{2}}.
             \end{equation*}
\end{itemize}
\end{theorem}

\begin{proof}
$\mathbf{(i)}_Y$ Because of Example  \ref{example:special_case} we may use Theorem \ref{theorem:main} to obtain for any 
$0 \le s < t \le T$ and any stopping time $\tau \in [s,t]$ that
 \equa
  {\st \frac{1}{c_{\eqref{theorem:main}}^p}} \E^{\ftn_\tau} |Y_t-Y_\tau|^p
	&\le& 2 c_{\eqref{example:special_case}}^p (t-\tau)^{p/2}
	    \left(1+\E \left[ \sup_{r \in [\tau,t]}|X_r|^p \mitt \cF_\tau \right]\right)
	+ \E^{\ftn_\tau} \left(\int_\tau^t |h(r,X_r,0,0)| dr\right)^{p} \\
	&&+ (t-\tau)^{p} \E^{\ftn_\tau}\left( g(X_T)^p + \left(\int_t^T |h(r,X_r,0,0)|dr\right)^p \right).
 \tion
Using $(A_g)$, $(A_h)$ and the fact
 \begin{equation}\label{equation:simppeli}
  \E^{\ftn_\tau} \sup_{\tau\le r \le T} |X_r|^p \le C^p(1+|X_\tau|^p),
 \end{equation}
where $C$ depends at most on $(T,p,K_{b,\sigma})$,
we may deduce
 \equa
  \E^{\ftn_\tau} |Y_t-Y_\tau|^p
	\le c_{\eqref{theorem:main_FBSDE}}^p (t-\tau)^{p/2} \left[ 1 + |X_\tau|^p \right],
 \tion
where $c_{\eqref{theorem:main_FBSDE}}>0$ depends at most on $(T,d,p,L_g,L_h,L_{b,\sigma},K_{b,\sigma},K_h)$.
Assertions $\mathbf{(i)}_Z$ and $\mathbf{(ii)}_Y$ follow analogously by applying Example \ref{example:special_case}, Theorem \ref{theorem:main}, and inequality \eqref{equation:simppeli}.
Assertion $\mathbf{(ii)}_Z$, on the other hand, follows directly from Example \ref{example:special_case} and Theorem \ref{theorem:main}.
\end{proof}

{\st One application} of Theorem \ref{theorem:main_FBSDE} are tail estimates of exponential type for $(Y,Z)$.
{\st In Theorem \ref{theorem:FBSDE_tail_Y} we treat the process $Y$ and in 
Theorem \ref{theorem:FBSDE_tail_Z} the  process $Z$. These theorems follow from Theorem \ref{theorem:main_FBSDE} using Theorem \ref{theorem:Stefan1}
analogous to the proof of Theorem \ref{theorem:tail-estimate}.}
\bigskip

\pagebreak

\begin{theorem}\label{theorem:FBSDE_tail_Y}
Assume that Assumption \ref{assumption:FBSDE} holds.
Then there exists an absolute constant {\st $c_0>0$} such that the following holds:
\begin{enumerate}
	\item[$\mathbf{(i)}$] For any $0 \le s < t \le T$ and any stopping time $\tau:\Om \to [s,t]$ we have
 \begin{equation*}
 \P_B\left( \sup_{u \in [\tau,t]} \frac{|Y_u-Y_\tau|}{c_{\eqref{theorem:main_FBSDE}}\sqrt{t-s}}
   > \lambda + {\st c_0} \mu \nu \right)
	\le e^{1-\mu} \P_B\left( \sup_{u \in [\tau,t]} \frac{|Y_u-Y_\tau|}{c_{\eqref{theorem:main_FBSDE}}\sqrt{t-s}} > \lambda \right)
	 + {\st c_0} \P_B\left( \sup_{u \in [\tau,t]} |X_u|^2 > \nu^2-1 \right)
 \end{equation*}
for all $\lambda,\mu,\nu>0$ and all $B \in \ftn_\tau$ of positive measure.
     \item[$\mathbf{(ii)}$] If additionally $(A_\sigma)$ holds, then we have
 \begin{multline*}
 \P_B\left( \sup_{u \in [\tau,t]} \frac{|Y_u-Y_\tau|}{d_{\eqref{theorem:main_FBSDE}}\sqrt{t-s}}
   > \lambda + {\st c_0} \mu \nu \right)\\
	\le e^{1-\mu} \P_B\left( \sup_{u \in [\tau,t]} \frac{|Y_u-Y_\tau|}{d_{\eqref{theorem:main_FBSDE}}\sqrt{t-s}} > \lambda \right)
	 + {\st c_0} \P_B\left( \sup_{u \in [\tau,t]} |X_u|^2(t-u) > \nu^2-1 \right)
 \end{multline*}
for all $\lambda,\mu,\nu>0$ and all $B \in \ftn_\tau$ of positive measure.
\end{enumerate}
\end{theorem}
\bigskip

\begin{theorem}\label{theorem:FBSDE_tail_Z}
Assume that Assumption \ref{assumption:FBSDE} holds. Then there exists an absolute constant {\st $c_0>0$} such that the following holds:
 \begin{enumerate}
 	\item[$\mathbf{(i)}$] For any
 	     $0 \le s < t \le T$ and any stopping time $\tau:\Om \to [s,t]$ we have
 	      \begin{multline*}
 	       \P_B \left( \sup_{u \in [\tau,t]}\left|\frac{\int_{\tau}^{u} Z_r dW_r}{C_{\eqref{theorem:main_FBSDE}} \beta_2\sqrt{t-s}}\right| > \lambda + {\st c_0}\mu\nu \right) \\
 	       \le e^{1-\mu} \P_B\left( \sup_{u \in [\tau,t]}\left|\frac{\int_{\tau}^{u} Z_r dW_r}{C_{\eqref{theorem:main_FBSDE}} \beta_2\sqrt{t-s}}\right| > \lambda \right)
 	       + {\st c_0} \P_B\left( \sup_{u \in [\tau,t]} |X_u|^2 > \nu^2-1 \right)
 	      \end{multline*}
 	      for all $\lambda,\mu,\nu>0$ and all $B \in \cF_\tau$ of positive measure, and where $\beta_2$ is the constant from Proposition \ref{proposition:BDG}.
 	\item[$\mathbf{(ii)}$] If additionally $(A_\sigma)$ holds, then
 	      \begin{equation*}
 	       \P_B \left( \sup_{u \in [\tau,t]}\left|\frac{\int_{\tau}^{u} Z_r dW_r}{D_{\eqref{theorem:main_FBSDE}} \beta_2\sqrt{t-s}}\right| > \lambda + {\st c_0}\mu \right) \\
 	       \le e^{1-\mu} \P_B\left( \sup_{u \in [\tau,t]}\left|\frac{\int_{\tau}^{u} Z_r dW_r}{D_{\eqref{theorem:main_FBSDE}} \beta_2\sqrt{t-s}}\right| > \lambda \right)	      
 	      \end{equation*}
 	     for all $\lambda,\mu>0$ and all $B \in \cF_\tau$ of positive measure, and where $\beta_2$ is the constant from Proposition \ref{proposition:BDG}.
 \end{enumerate}
\end{theorem}

One might ask if it is necessary to use the theory of \emph{weighted} $\bmo$ instead of non-weighted $\bmo$.
The following example shows that the weight processes of Theorem \ref{theorem:main_FBSDE} $\mathbf{(i)}_Y$ and $\mathbf{(i)}_Z$ are sharp:

\begin{example}\label{example_bounded_below2}
	Consider the FBSDE
	\equa
	 X_t &=& \int_0^t \sqrt{4e^{-s}+X_s^2} dW_s, \quad t \in [0,T], \\
	 Y_t   &=& X_T- \int_t^T Z_r dW_r, \quad t \in [0,T].
	\tion
	This FBSDE is of the same form as \eqref{equation:BSDE} with $d=1$, $b\equiv 0$, 
	$\sigma(t,x) = \sqrt{4e^{-t}+x^2}$, $h\equiv 0$ and $g(x)=x$, so that Assumption \ref{assumption:FBSDE} holds.
	Now we have for all {\st $p\in [2,\infty)$} and all $0 \le s < t \le T$ that
	\begin{equation*}
	\E \left[ |Y_t-Y_s|^p \mitt \cF_s \right]
	\ge \frac{(t-s)^{\frac{p}{2}} \left(1+|X_s|^p\right)}{e^{\frac{Tp}{2}}},
	\end{equation*}
	as well as
	\begin{equation*}
	\E \left[ \left(\int_{s}^t |Z_r|^2 dr \right)^{\frac{p}{2}} \mitt \cF_s \right]
	\ge \frac{(t-s)^{\frac{p}{2}}\left(1+|X_s|^p\right)}{\beta_p^p e^{\frac{Tp}{2}}} ,
	\end{equation*}
	where $\beta_p$ is the constant from Proposition \ref{proposition:BDG}.
\end{example}
\begin{proof}
	First note that $Y_t = X_t$, and that $X_t =  2\sinh(W_t)e^{-\frac{t}{2}} = e^{W_t - \frac{t}{2}} - e^{-W_t - \frac{t}{2}}.$
	Furthermore, we have the equalities:
	\equa
	 \E^{\cF_s} |e^{W_t - \frac{t}{2}} - e^{W_s - \frac{s}{2}}|^2
	 &=& |e^{W_s - \frac{s}{2}}|^2 (e^{t-s} - 1), \\
	 \E^{\cF_s} |e^{-W_t - \frac{t}{2}} - e^{-W_s - \frac{s}{2}}|^2
	 &=& |e^{-W_s - \frac{s}{2}}|^2 (e^{t-s} - 1), \\
	\E^{\cF_s} (e^{W_t - \frac{t}{2}} - e^{W_s - \frac{s}{2}})(e^{-W_t - \frac{t}{2}} - e^{-W_s - \frac{s}{2}})
	&=& e^{-t}(1-e^{t-s}), \\
   |X_s|^2 + 2e^{-s}
    &=& |e^{W_s - \frac{s}{2}}|^2 + |e^{-W_s - \frac{s}{2}}|^2, \\	 	
	\tion
	so that
	 \equa
	  \E^{\cF_s} |Y_t-Y_s|^2 &=& \E^{\cF_s} \left|e^{W_t - \frac{t}{2}}-e^{W_s - \frac{s}{2}} - \left( e^{-W_t - \frac{t}{2}} - e^{-W_s - \frac{s}{2}} \right)\right|^2 \\
	                         &=& |e^{W_s - \frac{s}{2}}|^2 (e^{t-s}-1)
	                           + |e^{-W_s - \frac{s}{2}}|^2 (e^{t-s}-1)
	                           - 2 e^{-t} (1-e^{t-s}) \\
	                         &=& (e^{t-s}-1) \left(|X_s|^2 + 2(e^{-t} + e^{-s}) \right) \\
	                         &\ge& (t-s) \left(|X_s|^2 + 1) \right) e^{-T}.
	 \tion
		Since $\frac{p}{2} \ge 1$, we also have
	\equa
	\E^{\ftn_s} |Y_t-Y_s|^p
	\ge \left(\E^{\ftn_s} |Y_t-Y_s|^2\right)^{\frac{p}{2}}
	\ge \left( e^{-T}(t-s) (1+|X_s|^2) \right)^{\frac{p}{2}}
	\ge e^{-\frac{Tp}{2}}(t-s)^{\frac{p}{2}}\left(1+|X_s|^p\right).
	\tion
	
	The result for the $Z$-process follows now immediately from
	\begin{equation*}
	\int_s^t Z_r dW_r = Y_t-Y_s.
	\end{equation*}
\end{proof}

The following example shows that the weight processes of Theorem \ref{theorem:main_FBSDE} $\mathbf{(ii)}_Y$ and $\mathbf{(ii)}_Z$ are sharp:

\begin{example}\label{example:bounded_below}
Consider the FBSDE
 \equa
  X_t &=& \int_0^t 1 dW_s, \quad t \in [0,T], \\
  Y_t &=& X_T + \int_t^T X_s ds - \int_t^T Z_s dW_s, \quad t \in [0,T],
 \tion
This FBSDE is of the same form as \eqref{equation:BSDE} with $d=1, b\equiv 0, \sigma \equiv 1$, $h(t,x,y,z)=x$, and $g(x)=x$, so that Assumptions \ref{assumption:FBSDE} and $(A_\sigma)$ hold.
Now we have for all {\st $p\in [2,\infty)$}, and all $0 \le s < t \le T$ that
 \begin{equation*}
  \E \left[ |Y_t - Y_s|^p \mitt \ftn_s \right] \ge (t-s)^{p/2}(1+|X_s|^p(t-s)^{p/2})
 \end{equation*}
as well as
 \begin{equation*}
  \E^{\cF_s} \left(\int_{s}^{t} |Z_r|^2 dr \right)^{\frac{p}{2}}
  \ge (t-s)^{\frac{p}{2}}.
 \end{equation*}
\end{example}
 \begin{proof}
We have for all $r \in [0,T]$ that
\begin{equation*}
  Y_r = \E^{\ftn_r} \left[ W_T + \int_r^T W_u du \right]
      = W_r(1+T-r),
 \end{equation*}
and therefore
 \equa
       \E^{\ftn_{s}} |Y_t-Y_s|^2
 &=&   \E^{\ftn_{s}} |W_t(1+T-t)-W_s(1+T-s)|^2 \\
 &=&   (t-s)(1+T-t)^2 + |W_s|^2(t-s)^2 \\
 &\ge& (t-s) (1 + |W_s|^2(t-s)).
 \tion
Since $\frac{p}{2} \ge 1$, we deduce
 \equa
       \E^{\ftn_{s}} |Y_t-Y_s|^p
\ge  \left[\E^{\ftn_{s}} |Y_t-Y_s|^2\right]^{p/2}
\ge  \left[(t-s) (1 + |W_s|^2(t-s))\right]^{p/2}
\ge  (t-s)^{p/2} (1 + |W_s|^p(t-s)^{p/2}).
 \tion
The result for the $Z$-process follows immediately from the fact that 
$Z_r = 1+(T-r)$.
\end{proof}


\section{Decoupling operators}\label{section:coupling}

We now recall the decoupling operators introduced in \cite{jossain}, as well as some of their properties proven there.
These operators are defined for random objects based on $\overline{\Om}$, see Section \ref{subsection:couplingsetting} below, but we will use them to deduce \emph{conditional estimates in the original probability space} $(\Om,\cF,\P)$.
These results are crucial in proving Theorem \ref{theorem:main}.

\subsection{Setting}\label{subsection:couplingsetting}

Recall the stochastic basis $(\Om,\ftn,\P,(\cF_t)_{t \in [0,T]})$ that was fixed in the beginning of Section \ref{section:preliminaries}.
Our fundamental random object is the Brownian motion $W=(W_t)_{t\in [0,T]}$, but for our decoupling technique we also need to have
a Brownian motion $W'$ that is independent of $W$. Thus we proceed as follows:

\begin{enumerate}[Step 1.]
 \item Fix another stochastic basis $(\Om',\ftn',\P',(\cF'_t)_{t \in [0,T]})$
{\st and a standard $d$-dimensional Brownian motion $W'=(W'_t)_{t\in [0,T]}$ that satisfy the same assumptions
as imposed on $(\Om,\ftn,\P,(\cF_t)_{t \in [0,T]},(W_t)_{t\in [0,T]})$ in Section \ref{section:preliminaries}.}

\item Let
\begin{equation*} \overline{\Omega} := \Omega\times\Omega', \quad
   \overline{\P} := \P\times \P', \quad
   \overline{\cF} := \overline{\cF\otimes\cF'}^{\overline{\P}}. \end{equation*}

\item Extend the Brownian motions $W$ and $W'$ canonically to $\overline{\Om}$, that is,
 \equa
   W(\om,\om') &:=& W(\om), \\
   W'(\om,\om') &:=& W'(\om').
 \tion
The augmented\footnote{Whenever we augment a filtration that is based on $\overline{\Om}$, we augment it by $\overline{\P}$-nullsets.}
natural filtration of the $2d$-dimensional Brownian motion $(W,W')$ is denoted by $\overline{\mathbb{F}}=(\overline{\cF}_t)_{t \in [0,T]}$.
\end{enumerate}

Hence, on the probability space $(\overline{\Om},\overline{\cF},\overline{\P})$, there are two independent
$d$-dimensional Brownian motions $W$ and $W'$.
Fix a Borel-measurable function $\vph:(0,T]\to [0,1]$. We define another standard $d$-dimensional Brownian motion on $(\overline{\Om},\overline{\cF},\overline{\P})$ by
  \begin{equation*} W^\vph_t: = \int_0^t \sqrt{1-\vph(u)^2} dW_u + \int_0^t \vph(u) dW_u', \quad t \in [0,T], \end{equation*}
and assume again continuity of all trajectories. The augmented natural filtration of $W^\vph$
is denoted by
$\mathbb{F}^\vph=(\ftn_t^\vph)_{t \in [0,T]}$
{\st and we obtain another stochastic basis
\[ (\overline{\Omega}, \cF_T^\varphi,\overline{\P},(\cF_t^\vph)_{t\in [0,T]}) \]
and can define, as in Notation \ref{notation},
\begin{equation}\label{eqn:definition_product_spaces}
   (\overline{\Om}_S,\Sigma^\varphi_S,\overline{\P}_S) :=
   \begin{cases}
   \left (\overline{\Om},  \cF_T^\varphi,\overline{\P}\right) & : S=0,\\
   \left ([0,T]\times \overline{\Om},{\mathcal B}([0,T])\otimes \cF_T^\varphi,\frac{\la}{T}\otimes\overline{\P}\right) & : S=T.
   \end{cases}
\end{equation}
Furthermore, we denote} the predictable $\sigma$-algebra on the stochastic basis
$(\overline{\Omega}, \cF_T^\varphi,\overline{\P},(\cF_t^\vph)_{t\in [0,T]})$ by $\cP^\vph$.
Denoting the function $\vph \equiv 0$ simply by $0$, we have that $W^0$ and (the extension of) $W$ are indistinguishable.
Since $\cF^0$ contains all $\overline{\P}$-nullsets, it follows that $(\cF^0_t)_{t \in [0,T]}$ and the augmentation of
$\sigma(W_r,r \in [0,t])_{t \in [0,T]}$ coincide.
Thus, we may agree to use $W^0$ for the extension of $W$, and similarly we use $W^1$ for the extension of $W'$.

\subsection{Decoupling operators}

Given a random variable $\xi$, whose randomness is given by $W$, we wish to define a random variable $\xi^\vph$ with the following two properties:
	\begin{itemize}
		\item[(1)] $\xi^\vph$ is a copy of $\xi$,
		\item[(2)] The randomness of $\xi^\vph$ is given by $W^\vph$.
	\end{itemize}
We accomplish this at the level of equivalence classes.
The fact that our procedure is well-defined is not proven here; all the proofs can be found in \cite{jossain}.
\begin{enumerate}[Step 1.]
 \item For $\xi \in \mathcal{L}_0(\Om,\cF,\P)$ take the canonical extension $\tilde \xi \in \mathcal{L}_0(\overline{\Om},\cF^0,\overline{\P})$,
       and let $[\xi] \in L_0(\overline{\Om},\cF^0,\overline{\P})$
       be the equivalence-class
       that contains all $\cF^0$-measurable random variables that are $\overline{\P}$-a.s. the same as $\tilde \xi$.

 \item We let $(h_k)_{k \in \N}$ be the ($L_2([0,T])$-normalized) Haar-functions on $[0,T]$, and denote by $W^0_{s,i}$ the $i$:th component of
       the Brownian motion $W^0$ for $i=1,\dots,d$.
       Now, letting $(g_n)_{n \in \mathbb{N}}:\overline{\Om}\to\R$ be the family of random variables $\int_0^T h_k(s) dW^0_{s,i}$
       where $i=1,\dots,d$ and $k \in \N$,       
       there exists a $\sigma(g_n,n\in\mathbb{N})$-measurable $\xi^0 \in [\xi]$.
  
 \item Defining $J:\overline{\Om} \to \R^{\mathbb{N}}, \quad J(\eta) = (g_n(\eta))_{n \in \mathbb{N}}$,
       there exists a random variable $\hat \xi:\R^{\mathbb{N}}\to\R$ such that $\xi^0$ can be factorized through $\R^\mathbb{N}$:
        \begin{equation*}
         \xi^0:\overline{\Om} \stackrel{J}{\to} \R^{\mathbb{N}} \stackrel{\hat \xi}{\to} \R.
        \end{equation*}

 \item Define $(g_n^\vph)_{n \in \N}$ analogously as
       $(g_n)_{n \in \N}$, using $W^\vph$ instead of $W^0$, and let $J^\vph:\overline{\Om} \to \R^{\mathbb{N}}$, 
       $J^\vph(\eta) =(g_n^\vph(\eta))_{n \in \mathbb{N}}$. Then it follows that $\hat\xi(J^\vph)$ is a
       well-defined $\sigma(g_n^\vph,n \in \mathbb{N})$-measurable random variable.

 \item Finally, we let $[\xi]^\vph \in L_0(\overline{\Om},\cF^\vph,\overline{\P})$ be the equivalence-class
       that contains all $\cF^\vph$-measurable random variables that are $\overline{\P}$-a.s. the same as $\hat\xi(J^\vph)$.
 \end{enumerate}

\begin{remark}\rm
$ $
\begin{itemize}
	\item[(1)] 
		Steps 2-5 yield the \emph{decoupling operator} $\C:L_0(\overline{\Om},\cF^0,\overline{\P}) \to L_0(\overline{\Om},\cF^\vph,\overline{\P})$ defined by
		\begin{equation*}
		\C([\xi]) = [\xi]^\vph.
		\end{equation*}
	In the following we will identify  $\xi$, $\tilde \xi$, and $[\xi]$, and denote all of them simply by $\xi$.
	Similarly, we will use the notation $\xi^\vph$ for both the equivalence class $[\xi]^\vph$, and any representative of it.
	
    \item[(2)]  The factorization and the approach used here is distributional,
    and does not require continuous paths or a gaussian distribution. As such, the approach might be useful also in other situations.
    
    \item[(3)] We can define $X^\vph$ for $X \in \cL_0(\Om_T,\Sigma_T,\P_T)$ analogously as above. The idea is that we change the randomness, but leave the time component unchanged.
    The point of defining this separately is to emphasize that $X^\vph \in L_0(\overline{\Om}_T,\Sigma_T^\vph,\overline{\P}_T)$, i.e. that representatives of $X^\vph$ are jointly measurable.
       
    \item[(4)] Our approach preserves continuity:
        Assume that $\M$ is locally $\sigma$-compact, $S \in \{0,T\}$, and \linebreak
    $f \in \cL_0(\Om_S,\Sigma_S,\P_S;C(\M))$. Then we may define $f^\vph \in L_0(\overline{\Om}_S,\Sigma_S^\vph,\overline{\P}_S;C(\M))$ by
    taking the continuous modification\footnote{Existence of such modification was proven in \cite[Proposition A.1]{jossain}.}
    of $(f(x)^\vph)_{x \in \M}$.
   \end{itemize}

\end{remark}

\subsection{Basic properties}

Predictability and adaptedness are transferred in the following sense:

\begin{proposition}[{\cite[Lemma 3.1 and {\st Proposition 2.12}]{jossain}}]
\label{proposition:measurability_preserved}
Let $\M$ locally $\sigma$-compact. Then the following holds true:
\begin{enumerate}[{\rm (i)}]
 \item If $\xi \in \mathcal{L}_0(\Om,\cF_t,\P)$ for some $t \in [0,T]$,
       then all representatives of 
       $\xi^\vph \in L_0(\overline{\Om},\cF^\vph,\overline{\P})$ are $\cF_t^\vph$-measurable.
 \item If $f \in \mathcal{L}_0(\Om_T,\mathcal{P},\P_T;C(\M))$\footnote{This means that $\eta \mapsto f(\eta,x)$ is $\mathcal{P}$-measurable for all $x \in \M$.},
       then there is a $\mathcal{P}^\vph$-measurable\footnote{This means that $\eta \mapsto f^\vph(\eta,x)$ is $\mathcal{P}^\vph$-measurable for all $x \in \M$.}
       representative of 
       \[ f^\vph \in L_0(\overline{\Om}_T,\Sigma_T^\vph,\overline{\P}_T;C(\M)). \]

 \item If $Y \in \mathcal{L}_0(\Om,\cF,\P;C([0,T]))$ is $(\cF_t)_{t \in [0,T]}$-adapted,
       then all representatives of 
       \[ Y^\vph \in L_0(\overline{\Om},\cF^\vph,\overline{\P};C([0,T])) \]
       are $(\cF_t^\vph)_{t \in [0,T]}$-adapted.
\end{enumerate}
\end{proposition}

We summarize some further properties proven in \cite{jossain}:

\begin{proposition}[{\cite[{\st Propositions 2.5, 2.13}, and Lemma 3.2]{jossain}}]
\label{proposition:coupling_properties}
Let {\st $N \ge 1$}, $S \in \{0,T\}$, $X,X_1,\dots,X_N \in \mathcal{L}_0(\Om_S,\Sigma_S,\P_S)$,
$Y \in \mathcal{L}_1(\Om_T,\Sigma_T,\P_T)$,
$g:\R^N \to \R$ be a Borel function, $f \in \mathcal{L}_0(\Om_S,\Sigma_S,\P_S;C(\R^N))$, and 
$Z \in \mathcal{L}_2(\Om_T,\cP,\P_T)$. Then the following holds true:
\begin{enumerate}[{\rm (i)}]
 \item $X \stackrel{d}{=} X^\vph$.
 \item $(g(X_1,\dots,X_N))^\vph = g(X_1^\vph,\dots,X_N^\vph)$.
 \item $(f(X_1,\dots,X_N))^\vph = f^\vph(X_1^\vph,\dots,X_N^\vph)$.
 \item $\left(\int_0^T Y(t)1_{\{\int_0^T |Y(s)|ds<\infty\}} dt\right)^\vph = \int_0^T Y^\vph(t)1_{\{\int_0^T |Y^\vph(s)|ds<\infty\}} dt$.
 \item $\left(\int_0^T Z(t) dW_t\right)^\vph = \int_0^T Z^\vph(t) dW_t^\vph$ for any predictable representative of
       $Z^\vph$.\footnote{By Proposition \ref{proposition:measurability_preserved}(ii) there exists such a representative.} 
 \item Let $X \in \cL_0(\Om_T,\Sigma_T,\P_T)$ and $Y \in \cL_0(\overline{\Om}_T,\Sigma_T^\vph,\overline{\P}_T)$. If there is a null-set
      $\cN \subseteq [0,T]$ with $Y(t) \in X(t)^\vph$ for all $t\in [0,T]\setminus \cN$, then $Y \in X^\vph$.
\end{enumerate}
\end{proposition}

Our next result can be interpreted as follows: if $(Y,Z)$ is a solution of an SDE, then $(Y^\vph,Z^\vph)$ is a solution of another SDE.
Note that we do not assume the SDEs to have unique solutions, we only assume that $(Y,Z)$ satisfies the equation.

\begin{proposition}[{\cite[Theorem 3.3]{jossain}}]
\label{proposition:coupling_SDE}
Assume that $f,g_i \in \mathcal{L}_0(\Om_T,\mathcal{P},\P_T;C(\R^{1+d}))$, $Z_i \in \cL_0(\Om_T,\mathcal{P},\P_T)$, $i=1,\dots,d$,
that $Y \in \cL_0(\Om,\cF,\P;C([0,T]))$ is $(\cF_t)_{t \in [0,T]}$-adapted, and that
 \begin{equation*}
  \E \left[ \int_0^T |f(r,Y_r,Z_r)|dr + \int_0^T |g(r,Y_r,Z_r)|^2 dr \right] < \infty.
 \end{equation*}
Furthermore, assume that $\xi \in \mathcal{L}_0(\Om,\cF,\P)$, and that equation
 \begin{equation}\label{SDE_to_change}
  Y_u = \xi + \int_u^T f(r,Y_r,Z_r)dr - \int_u^T g(r,Y_r,Z_r) dW_r, \quad u \in [0,T],
 \end{equation}
holds $\P$-almost surely.
If we fix any predictable representatives of $f^\vph,g_i^\vph,Z_i^\vph$, and an $(\cF_t^\vph)_{t \in [0,T]}$-adapted (continuous)
representative of $Y^{\vph}$, we have
 \begin{equation*}
  \E \left[ \int_0^T |f^\vph(r,Y_r^\vph,Z_r^\vph)|dr + \int_0^T |g^\vph(r,Y_r^\vph,Z_r^\vph)|^2 dr \right] < \infty,
 \end{equation*}
and we have that the equation
 \begin{equation}\label{SDE_changed}
  Y_u^\vph = \xi^\vph + \int_u^T f^\vph(r,Y_r^\vph,Z_r^\vph)dr - \int_u^T g^\vph(r,Y_r^\vph,Z_r^\vph) dW_r^\vph, u \in [0,T],
 \end{equation}
holds $\overline{\P}$-almost surely.
\end{proposition}

\subsection{Conditional results}
\label{sec:conditional_results}

From now on we will exclusively use functions $\vph$ of the form
 \begin{equation*}
  \chi_{(s,t]}:(0,T]\to[0,1], \quad \chi_{(s,t]}(r) =
   \begin{cases}
    1 , \mbox{ if } r \in (s,t], \\
    0 , \mbox{ if } r \not \in (s,t],
   \end{cases}
 \end{equation*}
where $0 \le s < t \le T$. 
\footnote{\st For a picture of the different Brownian motions $W$, $W'$, and $W^{(s,t]}$ see Figure 1.}
\begin{figure}[h]\label{figure1}
	\includegraphics[scale=0.5]{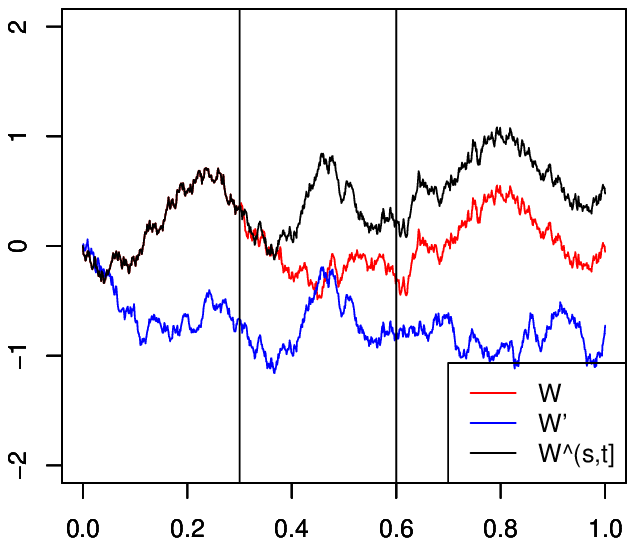}\vspace*{-2em}
	\caption{Brownian motions $W$,$W'$ and $W^{(s,t]}$. Here $s=0.3$, $t=0.6$ and $T=1$.}
\end{figure}
To keep the notation light, {\st we let 
\begin{equation}\label{eqn:xiab}
X^{(s,t]}:=X^{\chi_{(s,t]}}.
\end{equation}}
Recall that the random object $X^{(s,t]}$ is obtained by changing the underlying Brownian motion $W^0$ to an independent one on the interval $(s,t]$.
If $X$ is independent of $\sigma(W_r^0-W_s^0,r \in (s,t])$,
we ought to have $X^{(s,t]}=X$. Precisely in what sense this holds, is answered by the following proposition:

\begin{proposition}\label{lemma:adapted_extended}
Let $0 \le s < t \le T$, and define the sigma-algebras
 \begin{equation}\label{eqn:sigma_algebras_bar}
  \overline{\cG_s^t} := \sigma(W^0_r,r\in[0,s]) \vee \sigma(W^0_r-W^0_t, r\in[t,T])\vee\overline{\mathcal{N}} \sptext{1}{and}{1} 
  \overline{\cH_s^t} := \mathcal{B}([0,T]) \otimes \overline{\cG_s^t},
 \end{equation}
where $\overline{\mathcal{N}}$ are the $\overline{\P}$-nullsets.
Then
 \begin{enumerate}[{\rm (i)}]
  \item  $\E^{\overline{\cH_s^t}} X = \E^{\overline{\cH_s^t}} X^{(s,t]}$ for any $X \in L_1(\overline{\Om}_T,\Sigma_T^0,\overline{\P}_T)$,
  \item  $\E^{\overline{\cG_s^t}} \alpha = \E^{\overline{\cG_s^t}} \alpha^{(s,t]}$ for any $\alpha \in L_1(\overline{\Om},\cF^0,\overline{\P})$,
  \item  $\alpha = \alpha^{(s,t]}$ for any $\alpha \in L_0(\overline{\Om},\overline{\cG_s^t},\overline{\P})$,
  \item  $X \in X^{(s,t]}$ for any $X \in \mathcal{L}_0(\overline{\Om}_T,\overline{\cH_s^t},\overline{\P}_T)$,
  \item  $f \in f^{(s,t]}$ for any $f \in \mathcal{L}_0(\overline{\Om}_T,\overline{\cH_s^t},\overline{\P}_T;C(\M))$.
 \end{enumerate}
\end{proposition}
\begin{proof}
First of all we note that
 \equa
  \overline{\cG_s^t} &=& \sigma(W_r^{(s,t]}, r \in [0,s])
                   \vee \sigma(W_r^{(s,t]}-W_t^{(s,t]}, r \in [t,T]) \vee \overline{\mathcal{N}}. 
 \tion
Hence, similarly as in Proposition \ref{proposition:measurability_preserved}(i) (i.e. {\cite[Lemma 3.1]{jossain}}), we have that
if $\alpha \in \mathcal{L}_0(\overline{\Om},\overline{\cG_s^t},\overline{\P})$, then all representatives of $\alpha^{(s,t]}$
are $\overline{\cG_s^t}$-measurable.
\smallskip

\textbf{(i)} We need to prove that
 \begin{equation*} \int_A X d\overline{\P}_T = \int_A X^{(s,t]} d\overline{\P}_T \end{equation*}
for all $A \in \overline{\cH_s^t}$. By linearity of the decoupling operator
we may assume that $X\ge 0$, and it is enough to consider a generating $\pi$-system, so that we assume $A$ to be of the form
 \begin{equation*}
  (r,\widehat{W}) := (r,W^0_{s_1},\dots,W^0_{s_n},W^0_{t_1}-W^0_t,\dots,W^0_{t_m}-W^0_t) \in B_1 \times B_2,
 \end{equation*}
where $n,m \in \N$, $0 \le s_i \le s < t \le t_j \le T$, $0 \le r \le T$, and $B_1 \in \mathcal{B}([0,T])$, $B_2 \in \mathcal{B}(\R^{(n+m)d})$.
Letting $Y(r,\om) := \chi_{B_1\times B_2}(r,\widehat{W}(\om))$ we have $Y \in \mathcal{L}_\infty(\overline{\Om}_T,\overline{\cH_s^t},\overline{\P}_T)$,
and $Y \in Y^{(s,t]}$ because of Proposition \ref{proposition:coupling_properties}(ii) and (vi).
Thus, again using Proposition \ref{proposition:coupling_properties},
 \equa
  \int_A X d\overline{\P}_T
&=& \int_{\overline{\Om}_T} XY d\overline{\P}_T
= \int_{\overline{\Om}_T} (XY)^{(s,t]} d\overline{\P}_T
= \int_{\overline{\Om}_T} X^{(s,t]} Y^{(s,t]} d\overline{\P}_T \\
&=& \int_{\overline{\Om}_T} X^{(s,t]} Y d\overline{\P}_T
= \int_A X^{(s,t]} d\overline{\P}_T.
 \tion

\textbf{(ii)} Can be shown similarly as \textbf{(i)}.

\smallskip

\textbf{(iii)} First let $\alpha \in L_1(\overline{\Om},\overline{\cG_s^t},\overline{\P})$. Then we have that
$\E^{\overline{\cG_s^t}} \alpha^{(s,t]} = \alpha^{(s,t]}$,
but from \textbf{(ii)} we have that $\E^{\overline{\cG_s^t}} \alpha^{(s,t]} = \alpha$ as well.
For $\alpha \in L_0(\overline{\Om},\overline{\cG_s^t},\overline{\P})$
the claim follows from the fact that for all $N \in \mathbb{N}$
 \begin{equation*}
  (N\wedge\alpha\vee(-N))^{(s,t]} = N\wedge\alpha^{(s,t]}\vee(-N).
 \end{equation*}

\textbf{(iv)} If $X \in \mathcal{L}_0(\overline{\Om}_T,\overline{\cH_s^t},\overline{\P}_T)$,
then by Fubini's theorem $X(r) \in \mathcal{L}_0(\overline{\Om},\overline{\cG_s^t},\overline{\P})$ for all $r \in [0,T]$, so that \textbf{(iii)} implies that $X(r) \in X(r)^{(s,t]}$ for all $r \in [0,T]$.
Since $\overline{\cH_s^t} \subseteq \Sigma_T^{(s,t]}$, we have that $X \in \cL_0(\overline{\Om}_T,\Sigma_T^{(s,t]},\overline{\P}_T)$
so that the claim follows from Proposition \ref{proposition:coupling_properties}(vi).
\smallskip

\textbf{(v)}
Follows directly from \textbf{(iv)} and the definition of $f^{(s,t]}$.
\end{proof}
\medskip

We want to deduce conditional estimates for random variables based on the probability space $(\Om,\cF,\P)$ from estimates 
obtained using the decoupling operators.
{\st Recall, that 
$\cF^0$ and $\cF_s^0$ were defined in Section \ref{subsection:couplingsetting},
$\overline{\cG_s^t}$ and $\overline{\cG_s^t}$ by equations \eqref{eqn:sigma_algebras_bar}, 
$\cG_s^t$ and $\cG_s^t$ by \eqref{eqn:sigma_algebras},
and $\xi^{(a,b]}$ by \eqref{eqn:xiab}. The} following result is vital:

\begin{lemma}\label{lemma:conditional_equivalence}
Let {\st $p\in [1,\infty)$}, $0 \le s < t \le T$, and $\xi \in \mathcal{L}_p(\overline{\Om},\cF^0,\overline{\P})$. Then
 \begin{equation}\label{equation:conditional_coupling}
  \frac{1}{2^p} \E^{\overline{\cG_s^t}} |\xi-\xi^{(s,t]}|^p
\le  \E^{\overline{\cG_s^t}} |\xi-\E^{\overline{\cG_s^t}} \xi|^p
\le  \E^{\overline{\cG_s^t}} |\xi-\xi^{(s,t]}|^p.
 \end{equation}
\end{lemma}
\begin{proof}
We know from {\cite[\st Lemma 4.23]{jossain}} that for any $X \in \mathcal{L}_p(\overline{\Om},\cF^0,\overline{\P})$
 \begin{equation}\label{expectation_equation}
	\half \|X-X^{(s,t]}\|_p
 \le	\|X-\E^{\overline{\cG_s^t}} X \|_p
 \le	\|X-X^{(s,t]}\|_p.
 \end{equation}
Let $A \in \overline{\cG_s^t}$ such that $\overline{\P}(A)>0$.
Using Propositions \ref{proposition:coupling_properties}(ii) and \ref{lemma:adapted_extended}(iii) we have
$(\xi\chi_A)^{(s,t]} = \chi_A\xi^{(s,t]}$ and $\E^{\overline{\cG_s^t}} (\xi \chi_A) = \chi_A \E^{\overline{\cG_s^t}} \xi$,
so that applying equation (\ref{expectation_equation}) with $X=\xi \chi_A$ implies the claim.
\end{proof}

\begin{cor}\label{cor:vital}\rm
Let {\st $p\in [1,\infty)$, $0\le s < t \le T$,}
$\xi \in \mathcal{L}_p(\Om,\cF,\P)$ and $\Psi \in \mathcal{L}_0(\Om,\cF,\P)$, and denote their canonical extensions by $\tilde \xi, \tilde \Psi$, respectively.
Then
 \begin{equation*}
                  \E^{\cF_s^0} |\tilde \xi-\xi^{(s,t]}|^p      \le \tilde \Psi \quad
\Rightarrow \quad \E^{\cF_{s}}|\xi-\E^{\cG_s^t}\xi|^p \le \Psi \quad
\Rightarrow \quad \E^{\cF_{s}^0}|\tilde \xi-\xi^{(s,t]}|^p     \le 2^p\tilde \Psi.
 \end{equation*}
\end{cor}
\begin{proof}
We have that
 $\tilde \xi \in \mathcal{L}_p(\overline{\Om},\cF^0,\overline{\P})$, and that the canonical extension of $\E^{\cF_{s}}|\xi-\E^{\cG_s^t}\xi|^p$
is $\overline{\P}$-a.s. equal to $\E^{\cF_s^0} |\tilde \xi-\E^{\overline{\cG_s^t}} \tilde \xi|^p$.
Applying $\E^{\cF_s^0}$ on equation \eqref{equation:conditional_coupling}, we have $\overline{\P}$-a.s.
 \begin{equation}\label{equation:conditional_coupling2}
  \frac{1}{2^p} \E^{\cF_s^0} |\tilde \xi-\xi^{(s,t]}|^p
\le  \E^{\cF_s^0} |\tilde \xi-\E^{\overline{\cG_s^t}} \tilde \xi|^p
\le  \E^{\cF_s^0} |\tilde \xi-\xi^{(s,t]}|^p,
 \end{equation}
and the claim follows.
\end{proof}

The same idea applies also for the generator of a BSDE. However, the result corresponding to Lemma \ref{lemma:conditional_equivalence}
being technically involved, is proven in the appendix.

\begin{cor}\label{cor:equivalence_generator}
Assume that $f \in \mathcal{L}_0(\overline{\Om}_T,\Sigma_T^0,\overline{\P}_T;C(\R^{1+d}))$
satisfies {\rm (C1)} with $\theta = 1$
and {\rm (C2)} with $p=1$, let {\st $q \in [1,\infty)$} and $0 \le s < t \le T$. Moreover, let $\Psi \in \mathcal{L}_1(\Om,\cF,\P)$,
and denote its canonical extension by $\tilde \Psi$.
Then
\[  \E^{\cF_s} \left(\int_s^T \sup_{x\in \R^{d+1}} |f(r,x)-(\E^{\cH_s^t} f)(r,x)| dr \right)^q \le \Psi 
     \Rightarrow  \E^{\cF_s^0} \left(\int_s^T \sup_{x\in \R^{d+1}} |f(r,x)-f^{(s,t]}(r,x)| dr \right)^q \le 2^q \tilde \Psi, \]
and conversely,
\[ \E^{\cF_s^0} \left(\int_s^T \sup_{x\in \R^{d+1}} |f(r,x)-f^{(s,t]}(r,x)| dr \right)^q \le \tilde \Psi 
    \Rightarrow \E^{\cF_s} \left(\int_s^T \sup_{x\in \R^{d+1}} |f(r,x)-(\E^{\cH_s^t} f)(r,x)| dr \right)^q \le \Psi. \]
\end{cor}

\begin{proof}
The canonical extension of
$\E^{\cF_s} \left(\int_s^T \sup_{x\in \R^{d+1}} |f(r,x)-(\E^{\cH_s^t} f)(r,x)| dr \right)^q$
is $\overline{\P}$-a.s. equal to \linebreak
$\E^{\cF_s^0} \left(\int_s^T \sup_{x\in \R^{d+1}} |f(r,x)-(\E^{\overline{\cH_s^t}} f)(r,x)| dr \right)^q$,
so that the result follows by applying $\E^{\cF_s^0}$ to the conclusion of Proposition \ref{proposition:ugliness} with $u_1=s$ and $u_2=T$.
To apply Proposition \ref{proposition:ugliness}, we show that for all $R>0$ it holds
$\int_{\Om_T} \sup_{x \in \overline{B}(0,R)}|f(x)| d\P_T < \infty$, where $\overline{B}(0,R) \subseteq \R^{1+d}$ is the closed ball of radius $R$.
Indeed, it follows from {\rm (C2)} and {\rm (C3)} that
 \equa
     \int_{\Om_T} \sup_{(y,z) \in \overline{B}(0,R)}|f(r,\om,y,z)| d\P_T(r,\om)
&\le&  \int_{\Om_T} \sup_{(y,z) \in \overline{B}(0,R)}|f(r,\om,0,0)| + L_y|y| + L_z (1+|z|) |z| d\P_T(r,\om) \\
&\le&  \E \left(\int_0^T |f(r,0,0)| dr \right) + L_y R + L_z (1+R) R < \infty.
 \tion
\end{proof}


\section{\st Proof of Theorem \ref{theorem:main} and {\st Example \ref{example:special_case}}}
\label{sec:remaining_proofs_based_on_decoupling} 

{\st
Again we use $\cF_u^0$ from Section \ref{subsection:couplingsetting} and
$\xi^{(a,b]}$ from \eqref{eqn:xiab}.}
The following is the counterpart {\st to} Assumption \ref{assumption:BMO}:
\smallskip

\begin{assumption}\label{assumption:BMO_product}
{\st Let $p\in [2,\infty)$ and $0 \le s < t \le T$.} There 
{\st are non-negative c\`adl\`ag processes $(\weight^\xi_{p,s,u,t})_{u \in [s,t]}$ and $(\weight^f_{p,s,u,t})_{u \in [s,t]}$,
such that  $((\weight^\xi_{p,s,u,t})^p)_{u \in [s,t]}$ and $((\weight^f_{p,s,u,t})^p)_{u \in [s,t]}$
are $(\cF_r)_{r \in [0,T]}$-supermartingales,}
whose canonical extensions $(\tilde{\weight}^\xi_{p,s,u,t})_{u \in [s,t]}$ and 
$(\tilde{\weight}^f_{p,s,u,t})_{u \in [s,t]}$ satisfy,
for any $u \in [s,t]$,
\begin{enumerate}
\item [{\rm ($\overline{{\rm C5}}$)}] ${\st \left ( \E^{\cF_u^0}|\xi-\xi^{(u,t]}|^p \right )^\frac{1}{p}}	\le \tilde{\weight}^\xi_{p,s,u,t}$,
\item [{\rm ($\overline{{\rm C6}}$)}] ${\st \left ( \E^{\cF_u^0}\left(\int_u^T \sup_{y,z} |f(r,y,z)-f^{(u,t]}(r,y,z)| dr\right)^p \right )^\frac{1}{p}}\le\tilde{\weight}^f_{p,s,u,t}$.
\end{enumerate}
\end{assumption}

\begin{remark}
\label{remark:equivalent_assumptions}
\rm 
It is immediate from Corollaries \ref{cor:vital} and \ref{cor:equivalence_generator}
that Assumptions \ref{assumption:BMO} and \ref{assumption:BMO_product} are equivalent
{\st in the sense, that when passing from one assumption to the other one can use the same weights multiplied
by the factor 2.}
\end{remark}

\subsection{Proof of Theorem \ref{theorem:main}}\label{subsection:theorem:main}

In this Section we deduce upper bounds for $\E^{\cF_\tau} |Y_t-Y_\tau|^p$ and  $\E^{\cF_\tau}\left(\int_{\tau}^{t} |Z_r|^2 dr\right)^{\frac{p}{2}}$, where $\tau:\Om \to [s,t]$ is any stopping time, and $0 \le s < t \le T$ are such that Assumption \ref{assumption:BMO} is satisfied.

\medskip

Our procedure consists of the following steps:
\begin{enumerate}[Step 1:]
 \item Let $0 \le s < t \le T$, $u \in [s,t]$, and consider the decomposition
	\begin{equation}\label{equation:decomposition}
         \left(\E^{\ftn_u} |Y_t - Y_u|^p\right)^{1/p}
         \le \left(\E^{\ftn_u} |Y_t - \E^{\ftn_u} Y_t|^p\right)^{1/p} + \left(\E^{\ftn_u} |Y_u - \E^{\ftn_u} Y_t|^p\right)^{1/p}
         =:  I_1^{1/p}+I_2^{1/p}.
	\end{equation}
 \item With the assumptions of Theorem \ref{theorem:main}, Proposition \ref{proposition:I_1} 
       together with Corollary \ref{cor:vital} implies
  \begin{equation*}
   I_1 +  \E^{\cF_u}\left(\int_{u}^{t} |Z_r|^2 dr\right)^{\frac{p}{2}}
   \le c_{\eqref{proposition:I_1}}^p 2^p
      {\st \left ( \weight_{p,s,u,t}^{\xi,f} \right )^p},
  \end{equation*}
  where $c_{\eqref{proposition:I_1}}>0$ depends at most on $(T,d,p,L_y,L_z,(s_N)_N)$.
 \item With the assumptions of Theorem \ref{theorem:main}, Proposition \ref{lemma:I_2} implies that
  \begin{equation*}
   I_2 \le c_{\eqref{lemma:I_2}}^p \weight_{p,s,u,t}^p,
  \end{equation*}
  where $c_{\eqref{lemma:I_2}}>0$ depends at most on $(T,d,p,L_y,L_z,(s_N)_N)$.
 \item In the end we extend the result from all deterministic times $u \in [s,t]$ to all stopping times \linebreak
 $\tau:\Om \to [s,t]$.
\end{enumerate}
\bigskip

The next Proposition is a conditional version of \cite[{\st Theorem 6.3}]{jossain}. Note that Assumption \ref{assumption:BMO} is not needed for this result.

\begin{proposition}\label{proposition:I_1}
Assume {\rm(C1)-(C4)} for $\theta \in [0,1]$ and $p \in [2,\infty) \cap (p_{{\rm(C4)}},\infty)$, and fix $0 \le u < t \le T$.
Then there exists $c_{\eqref{proposition:I_1}} > 0$ depending at most on $(T,d,p,L_y,L_z,(s_N)_N)$ such that
 \equa
&&\E^{\cF_u^0} \sup_{r \in [u,T]} |Y_r^{(u,t]} - Y_r|^p
+ \E^{\cF_u^0} \left[
     \left(\int_u^t |Z_r|^2 dr \right)^{\frac{p}{2}}
   + \left(\int_u^T |Z_r^{(u,t]} - Z_r|^2 dr \right)^{\frac{p}{2}}
               \right] \\
&\le& c_{\eqref{proposition:I_1}}^p
      \E^{\cF_u^0} \left( |\xi^{(u,t]} - \xi|
       + \int_u^T |f(r,Y_r,Z_r) - f^{(u,t]}(r,Y_r,Z_r)| \right)^p.
 \tion
\end{proposition}
\begin{proof}
Let $A^0 \in \cF_u^0$ such that $\overline{\P}(A^0)>0$. Since the $\sigma$-algebras $\cF_u^0$ and $\cF_u \otimes \{\emptyset, \Om'\}$ differ only by null-sets, 
it follows that there exists $A \in \cF_u$ with $\P(A)>0$ such that
$\overline{\P} \left( 1_{(A\times\Om')} = 1_{A^0} \right) = 1$.
Now we define
 \equa
  \overline{\xi}      &:=& (\xi - Y_u) 1_A, \\
  \overline{f}(r,y,z) &:=& f(r,y+Y_u,z)1_A 1_{(u,T]}(r), \\
  \overline{Y}_r      &:=& (Y_r-Y_u) 1_A 1_{(u,T]}(r), \\
  \overline{Z}_r      &:=& Z_r 1_A 1_{(u,T]}(r).
 \tion
Note that $\overline{f}$ is designed to satisfy for all $r \in [0,T]$ the equation
 \begin{equation*}
  \overline{f}(r,\overline{Y}_r,\overline{Z}_r)
= f(r,Y_r,Z_r) 1_A 1_{(u,T]}(r).
 \end{equation*}
It is straight-forward to check that since $(f,Y,Z)$ satisfy (C1)-(C4), also $(\overline{f},\overline{Y},\overline{Z})$ satisfy (C1)-(C4). Moreover, $(t,\om) \mapsto \overline{f}(t,\om,y,z)$ is predictable for all $(y,z) \in \R^{1+d}$.
Now we have that $(\overline{Y},\overline{Z})$ is a solution {\st to}
\begin{equation*}
 \tilde Y_t = \overline{\xi} + \int_t^T \overline{f}(r,\tilde Y_r,\tilde Z_r) dr - \int_t^T \tilde Z_r dW_r.
\end{equation*}
Since $(\overline{f},\overline{Y},\overline{Z})$ satisfy conditions (C1)-(C4), and because of Lemma \ref{lemma:Ynice}, it follows that they also 
satisfy the assumptions of \cite[{\st Theorem 6.3}]{jossain}.
Applying \cite[{\st Theorem 6.3}]{jossain} with $\psi := 0$, and $\vph := 1_{(u,t]}$ implies that there exists $c_{\st (6.3)}>0$ depending at most on $(T,d,p,L_y,L_z,(s_N)_N)$ such that
 \equa
&&\left\| \sup_{r \in [u,T]}
     |\overline{Y}_r^{(u,t]} - \overline{Y}_r| \right\|_p
+ \left\|\left( \int_u^t |\overline{Z}_r|^2 dr \right)^{\half}\right\|_p
+ \left\|\left( \int_u^T |\overline{Z}^{(u,t]}_r - \overline{Z}_r|^2 dr \right)^{\half}\right\|_p \\
&\le& c_{\st (6.3)} \left[ \left\|\overline{\xi}^{(u,t]} - \overline{\xi}\right\|_p
 + \left\| \int_u^T |\overline{f}^{(u,t]}(r,\overline{Y}_r,\overline{Z}_r)
 - \overline{f}(r,\overline{Y}_r,\overline{Z}_r)| dr \right\|_p
  \right].
 \tion
By definitions of $(\overline{\xi},\overline{Y},\overline{Z},\overline{f})$ and using properties of the decoupling operators, in particular note that \linebreak
$1_{A^0}^{(u,t]} = 1_{A^0}$ since $A^0 \in \cF_u^0$, this reads as
 \equa
&& \left\| \sup_{r \in [u,T]}
 |Y_r^{(u,t]} - Y_r| 1_{A^0}\right\|_p
+ \left\|\left( \int_u^t |Z_r|^2 dr \right)^{\half} 1_{A^0} \right\|_p
+ \left\|\left( \int_u^T |Z^{(u,t]}_r - Z_r|^2 dr \right)^{\half} 1_{A^0} \right\|_p \\
&\le& c_{\st (6.3)} \left[ \left\| (\xi^{(u,t]} - \xi) 1_{A^0} \right\|_p
 + \left\| \int_u^T |f^{(u,t]}(r,Y_r,Z_r) - f(r,Y_r,Z_r)| dr 1_{A^0} \right\|_p \right],
 \tion
which immediately implies the claim.
\end{proof}

Next we try to find an upper bound for $I_2 = \E^{\ftn_u} |Y_u - \E^{\ftn_u} Y_t|^p$. We accomplish this by upper bounding $\E^{\cF_u}|\int_u^t f(r,Y_r,Z_r)dr|^p$.
First we have a simple upper bound for the $Y$-term, given in terms of the data $(\xi,f)$.

\begin{lemma}\label{lemma:apriori_Y}
Assume {\rm (C1)-(C4)} for $\theta \in [0,1]$ and $p \in [2,\infty) \cap (p_{(C4)},\infty)$. Then we have for any $u \in [0,T]$ that
\begin{equation*}
    \E^{\cF_u} \left(\sup_{r \in [u,T]} |Y_r|
+  \left(\int_u^T |Z_r|^2 dr \right)^{\frac{1}{2}} \right)^p
\le c_{\eqref{lemma:apriori_Y}}^p \E^{\cF_u} \left(|\xi| + \int_u^T |f(r,0,0)|dr\right)^p,
\end{equation*}
where $c_{\eqref{lemma:apriori_Y}}>0$ depends at most on $(T,d,p,L_y,L_z,(s_N)_N)$.
\end{lemma}
\begin{proof}
Let $A \in \cF_u$, and put
\equa
\xi^0       &=& (\xi - Y_u) 1_A, \\
f^0(r,y,z)  &=& f(r,y+Y_u,z) 1_A 1_{(u,T]}(r), \\
Y^0_r       &=& (Y_r - Y_u) 1_A 1_{(u,T]}(r), \\
Z^0_r       &=& Z_r 1_A 1_{(u,T]}(r),
\tion
as well as
$\xi^1 = 0$, 
$f^1(r,y,z) = 0$, 
$Y^1_r      = 0$, 
$Z^1_r       = 0$.  
As in the proof of Proposition \ref{proposition:I_1}, we have that $(f^0,Y^0,Z^0)$ satisfy {\rm (C1)-(C4)}. 
This yields the assumptions of \cite[{\st Lemma 5.26}]{jossain}, which immediately implies the claim.
\end{proof}

Next we deduce the desired upper bound for $I_2$.

\begin{proposition}\label{lemma:I_2}
Assume {\rm (C1)}-{\rm (C4)} for $\theta \in [0,1]$ and $p \in [2,\infty) \cap (p_{(C4)},\infty)$,
and let $0 \le s < t \le T$ such that {\rm (C5)} and {\rm (C6)} are satisfied. Then we have for any $u \in [s,t]$ that
\[ |Y_u - \E^{\ftn_u} Y_t|^p 
    \le  c_{\eqref{lemma:I_2}}^p \bigg[
  {\st \left ( \weight_{p,s,u,t}^{\xi,f}\right)^p}
+ \E^{\cF_u} \left( \int_u^t |f(r,0,0)|dr \right)^p
	 +   (t-u)^p \E^{\cF_u} \left(|\xi| + \int_u^T |f(r,0,0)|dr\right)^p \bigg], \]
where $c_{\eqref{lemma:I_2}}>0$ depends at most on $(T,d,p,L_y,L_z,(s_N)_N)$.
\end{proposition}
\begin{proof}
We have directly
 \equa
        |Y_u-\E^{\cF_u} Y_t|^p
    &=&  \left| \E^{\cF_u} \int_u^t f(r,Y_r,Z_r) dr \right|^p \\
    &\le& \E^{\cF_u} \left|\int_u^t |f(r,0,0)| + L_y|Y_r| + L_z[1+|Z_r|]^\theta |Z_r| dr \right|^p \\
    &\le& C_p \E^{\cF_u} \bigg[ \left(\int_u^t |f(r,0,0)|dr \right)^p
                              + L_y^p(t-u)^p \sup_{r \in [u,t]}|Y_r|^p \\
                            &&
                            + L_z^p  \left(\int_u^t |Z_r| dr \right)^p
                              + L_z^p  \left(\int_u^t |Z_r|^{1+\theta} dr \right)^p\bigg].
 \tion
Lemma \ref{lemma:apriori_Y} gives us an upper bound for the second term.
For the third term we may apply Proposition \ref{proposition:I_1} and {\st Assumption \ref{assumption:BMO_product} 
with Remark \ref{remark:equivalent_assumptions}} to deduce
 \equa
   \E^{\cF_u} \left(\int_u^t |Z_r| dr \right)^p
 \le (t-u)^{p/2}  \E^{\cF_u} \left(\int_u^t |Z_r|^2 dr \right)^{p/2}
 \le (2c_{\eqref{proposition:I_1}})^p (t-u)^{p/2}  {\st \left(\weight^{\xi,f}_{p,s,u,t}\right)^p}.
 \tion
For the last term we use also Proposition \ref{proposition:Fefferman_cond_new}
and Assumption (C4) to deduce
 \equa
  \E^{\ftn_u} \left( \int_u^t |Z_r|^{1+\theta} dr \right)^p
	&\le& c_p \sup_{r \in [u,t]} \left\|\E^{\cF_r} \int_r^t |Z_v|^{2\theta} dv \right\|_\infty^{\frac{p}{2}}
				\E^{\ftn_u} \left( \int_u^t |Z_r|^2 dr \right)^{\frac{p}{2}} \\
	&\le& c_p \| \chi_{(u,t]} |Z|^\theta \|_{\bmo(S_2)}^p (2c_{\eqref{proposition:I_1}})^p {\st \left( \weight_{p,s,u,t}^{\xi,f}\right)^p}.
 \tion
\end{proof}
\bigskip

\textbf{Proof of Theorem \ref{theorem:main}:}

Assume that (C1)-(C6) hold for $\theta \in [0,1]$, $p \in [2,\infty)\cap(p_{{\rm (C4)}},\infty)$, and $0 \le s < t \le T$.
\smallskip

$\mathbf{(i)}$ It follows {\st from} Propositions \ref{proposition:I_1} and \ref{lemma:I_2} that there exists {\st a constant} $C>0$
depending at most on $(T,d,p,L_y,L_z,(s_N)_N)$, such that for all $u \in [s,t]$
\begin{equation*}
	\E^{\ftn_u} |Y_t-Y_u|^p \le C\weight_{p,s,u,t}^p.
\end{equation*}
Since {\st $((\weight_{p,s,u,t}^\xi)^p)_{u \in [s,t]}$ and $((\weight_{p,s,u,t}^f)^p)_{u \in [s,t]}$} are supermartingales,
it follows that $(\weight_{p,s,u,t}^p)_{u \in [s,t]}$ as well is a supermartingale.
Applying \cite[Theorem 3.13, page 16]{Karatzas:Shreve} on $\E^{\cF_u} \int_u^t |f(r,0,0)|dr$, we deduce that $(\weight_{p,s,u,t}^p)_{u \in [s,t]}$ has a c\`adl\`ag modification, to which we will switch without changing the notation. Applying Lemma \ref{lemma:to_stopping_times} with
$
  \alpha_u       := |Y_t-Y_u|^p$ and $\weight_u := C\weight_{p,s,u,t}^p
$
implies the claim.

\smallskip
$\mathbf{(ii)}$ It follows from Proposition \ref{proposition:I_1} that there exists $C>0$
depending at most on $(T,d,p,L_y,L_z,(s_N)_N)$, such that for all $u \in [s,t]$ we have
\begin{equation*}
	  \E^{\ftn_u} \left(\int_u^t |Z_r|^2 dr \right)^{\frac{p}{2}}
\le {\st C^p \left(\weight_{p,s,u,t}^{\xi,f} \right )^p}.
\end{equation*}
Hence, the claim follows by applying Lemma \ref{lemma:to_stopping_times} with
$\alpha_u      := \left(\int_u^t |Z_r|^2 dr \right)^{\frac{p}{2}}$ and $\weight_u := {\st C^p \left(\weight_{p,s,u,t}^{\xi,f} \right )^p}$.
\qed

\subsection{\st Proof of Example \ref{example:special_case}}
\label{subsec:proof_lemma:special_case}

{\st We start with an inequality, which} proof is the same as that of \cite[Theorem 2.5]{GGG:12}.
{\st To do so, recall that $(\overline{\cF}_t)_{t\in [0,T]}$ is the natural augmented filtration of $(W,W')$.} 

\begin{lemma}\label{lemma:BMO_X}
Assume $(A_{b,\sigma})$, let $0 \le s < t \le T$ and {\st $p\in [2,\infty)$}. Then there exists $C_{\eqref{lemma:BMO_X}}>0$ depending at most on $(T,d,p,L_{b,\sigma},K_{b,\sigma})$
such that
 \begin{equation*} \E \left[ \sup_{r \in [s,T]} |X_r^{(s,t]}-X_r|^p \mitt \overline{\cF}_s \right]
		\le C_{\eqref{lemma:BMO_X}}^p (t-s)^{p/2} \left(1 + \E \left[ \sup_{r \in [s,t]}|X_r|^p \mitt \overline{\cF}_s \right]\right). \end{equation*}
If additionally $(A_\sigma)$ holds, then  there exists $D_{\eqref{lemma:BMO_X}}>0$ depending at most on $(T,d,p,L_{b,\sigma},K_\sigma)$
such that
 \begin{equation*} \E \left[ \sup_{r \in [s,T]} |X_r^{(s,t]}-X_r|^p \mitt \overline{\cF}_s \right]
		\le D_{\eqref{lemma:BMO_X}}^p (t-s)^{p/2}. \end{equation*}
\end{lemma}
\begin{proof}
Using Proposition \ref{proposition:coupling_SDE} we have
 \equa
  X_r^{(s,t]}-X_r &=& \int_s^r \left( b(u,X_u^{(s,t]}) - b(u,X_u)\right) du
		 +  \int_s^r \sigma(u,X_u^{(s,t]}) 1_{(s,t]}(u) dW^1_u
		  - \int_s^r \sigma(u,X_u) 1_{(s,t]}(u) dW^0_u \\
		&&+ \int_s^r \left( \sigma(u,X_u^{(s,t]}) - \sigma(u,X_u) \right) (1-1_{(s,t]}(u)) dW^0_u
 \tion
for all $r \in [s,T]$, $\overline{\P}$-a.s.
Next we let $A \in \overline{\cF}_s$ with $\overline{\P}(A)>0$, and define $g:[s,T] \to [0,\infty)$ by
 \begin{equation*} g(v) := \E \left( \sup_{s \le r \le v} |X_r^{(s,t]}-X_r|^p 1_A \right)
	  = \int_A \sup_{s \le r \le v} |X_r^{(s,t]}-X_r|^p d\overline{\P}. \end{equation*}
Using basic inequalities,
we have for all $v \in [s,T]$ that
 \begin{eqnarray}\label{equation:constanttiini}
  g(v)
	&=& \E \sup_{s \le r \le v} \bigg| \int_s^r \left( b(u,X_u^{(s,t]}) - b(u,X_u)\right) 1_A du
	  + \int_s^r \sigma(u,X_u^{(s,t]}) 1_{(s,t]}(u) 1_A dW^1_u \nonumber\\
	&&- \int_s^r \sigma(u,X_u) 1_{(s,t]}(u) 1_A dW^0_u
	  + \int_s^r \left( \sigma(u,X_u^{(s,t]}) - \sigma(u,X_u) \right) 1_{(s,t]^c}(u) 1_A dW^0_u\bigg|^p \\
&\le&	C(t-s)^{p/2} \int_A 1 + \sup_{s \le r \le t}|X_r|^p d\overline{\P} + C \int_s^v g(u) du, \nonumber
 \end{eqnarray}
where the constant $C$ depends at most on $(T,d,p,L_{b,\sigma},K_{b,\sigma})$. Then
it follows from Gronwall's lemma that
\begin{equation*}\label{equation:nice}
 \int_A \sup_{s \le r \le T} |X_r^{(s,t]}-X_r|^p d\overline{\P} = g(T) \le C_{\eqref{lemma:BMO_X}}^p(t-s)^{p/2} \int_A 1 + \sup_{s \le r \le t}|X_r|^p d\overline{\P},
\end{equation*}
where the constant $C_{\eqref{lemma:BMO_X}}$ depends at most on $(T,d,p,L_{b,\sigma},K_{b,\sigma})$.
If $(A_\sigma)$ holds, then we can deduce from Equation \eqref{equation:constanttiini} that
\begin{equation*}
  g(v)
\le	C(t-s)^{p/2}\overline{\P}(A) + C \int_s^v g(u) du,
\end{equation*}
where the constant $C$ now depends at most on $(T,d,p,L_{b,\sigma},K_\sigma)$. The result again follows from Gronwall's lemma.
\end{proof}

{\bf Proof of Example \ref{example:special_case}:}

(${\bf C1}$)-(${\bf C4}$): Follow from {\cite[Theorem 4.2]{Briand:02}}, since $(A_h)$ implies that (C1) holds with $\theta=0$,
		  $(A_{b,\sigma})$ together with $(A_g)$ implies $\E|g(X_T)|^p<\infty$,
		  and $(A_{b,\sigma})$ together with $(A_h)$ implies $\E \left(\int_0^T |h(r,X_r,0,0)| dr\right)^{p} < \infty$. \\

($\overline{{\bf C5}}$): Let $0 \le s \le u \le t \le T$. Using Proposition \ref{proposition:coupling_properties}(ii) we have that
 \begin{equation*} (g(X_T))^{(u,t]}=g(X_T^{(u,t]}), \end{equation*}
and Proposition \ref{proposition:coupling_SDE} implies that $X^{(u,t]}$ is the solution of
\begin{equation*} X_r^{(u,t]} = x + \int_0^r b(v,X_v^{(u,t]})dv
                    + \int_0^r \sigma(v,X_v^{(u,t]})dW_v^{(u,t]}, \quad r \in [0,T]. \end{equation*}
It follows from $(A_g)$ that
\begin{equation*}
  \E^{\cF_u^0} |g(X_T)-g(X_T^{(u,t]})|^p
  \le L_g^p \E^{\cF_u^0} |X_T-X_T^{(u,t]}|^p.
\end{equation*}
Finally, Lemma \ref{lemma:BMO_X} implies
 \begin{equation*}
      \E^{\cF_u^0} |X_T-X_T^{(u,t]}|^p
  \le C_{\eqref{lemma:BMO_X}}^p (t-u)^{p/2} \left(1+\E \left[ \sup_{r \in [u,t]}|X_r|^p \mitt \cF^0_u \right]\right), 
 \end{equation*}
and if $(A_\sigma)$ holds, then Lemma \ref{lemma:BMO_X} implies
 \begin{equation*}
      \E^{\cF_u^0} |X_T-X_T^{(u,t]}|^p
  \le D_{\eqref{lemma:BMO_X}}^p (t-u)^{p/2}. 
 \end{equation*}

($\overline{{\bf C6}}$): Let $0 \le s \le u \le t \le T$. We notice that Proposition \ref{proposition:coupling_properties} implies
 \begin{equation*} (h(r,X_r,y,z))^{(u,t]} = h(r,X_r^{(u,t]},y,z). \end{equation*}
The result again follows from Lemma \ref{lemma:BMO_X}, since $(A_h)$ implies
\begin{equation*}
  \E^{\cF_u^0} \left(\int_0^T \sup_{y,z} |h(r,X_r,y,z)-h(r,X_r^{(u,t]},y,z)|dr\right)^{p}
  \le \E^{\cF_u^0} \left(\int_0^T L_h |X_r^{(u,t]}-X_r| dr\right)^{p}.
 \end{equation*}
For all $u \in [s,t]$ we let
 \begin{equation*}
  \weight_u := \weight_{p,s,u,t} := (t-u)^{p/2} \left(1+\E \left[ \sup_{r \in [u,t]}|X_r|^p \mitt \cF_u \right]\right),
 \end{equation*}
and get that the process $(\weight_u)_{u \in [s,t]}$ is a supermartingale. Since
$u \mapsto \E \weight_u$ is continuous, there exists a c\`adl\`ag modification of the process $(\weight_r)_{r \in [s,t]}$.
This modification is a c\`adl\`ag supermartingale, and for any fixed $u \in [s,t]$ its canonical extension
coincides $\overline{\P}$-a.s. with
 \begin{equation*}
  (t-u)^{p/2} \left(1+\E \left[ \sup_{r \in [u,t]}|X_r|^p \mitt \cF_u^0 \right]\right).
 \end{equation*}
Hence, there exists $C>0$ depending at most on $(T,d,p,L_g,L_h,L_{b,\sigma},K_{b,\sigma})$
such that Assumption \ref{assumption:BMO_product} holds
for all $0 \le s < t \le T$ with
 \begin{equation*}
 {\st  \left (\tilde{\weight}^f_{p,s,u,t} \right )^p = \left ( \tilde{\weight}_{p,s,u,t}^\xi\right )^p} := C^p (t-u)^{p/2} \left(1+\E \left[ \sup_{r \in [u,t]}|X_r|^p \mitt \cF_u^0 \right]\right).
 \end{equation*}
If additionally $(A_\sigma)$ holds, then there exists $D>0$ depending at most on $(T,d,p,L_g,L_h,L_{b,\sigma},K_\sigma)$
such that we may choose
 \begin{equation*} {\st \tilde{\weight}^f_{p,s,u,t} = \tilde{\weight}_{p,s,u,t}^\xi=D(t-u)^{1/2}}. \end{equation*}
\qed


{\st
\section{Some Applications}
\label{section:applications}

In this section we discuss some applications of the tail estimates obtained in Theorem \ref{theorem:FBSDE_tail_Y}.
We can use them in two different ways: Firstly, we can exploit 
the tail estimates (Sections \ref{sec:spline_approximation} and \ref{sec:direct_simulation}), 
secondly we may exploit the fact that we can control {\em all conditional moments} which might allow
us for a change of the underlying measure (Section \ref{sec:change_of_measure}).

\subsection{Uniform spline approximation of the process $Y$}
\label{sec:spline_approximation}

To get a path-dependent approximation of the process $Y=(Y_t)_{t\in [0,T]}$ based on a method that provides  
approximations $\widehat Y_{t_i}$ of $Y_{t_i}$ for some  deterministic time-net
$\pi=(t_i)_{i=0}^n$, $0=t_0 < \cdots < t_n =T$,  one can consider a linear spline 
\[\widehat  Y_t^\pi := (1-\theta) \widehat Y_{t_{i-1}} + \theta \widehat Y_{t_i}
   \sptext{1}{for}{1} t\in I_i^\pi := [t_{i-1},t_i] \sptext{.5}{with}{.5}
   t=(1-\theta)t_{i-1} + \theta t_i. \]
We get that 
\[ \sup_{i=0,\ldots,n} |\widehat{Y}_{t_i}^\pi-Y_{t_i}|
   \le \| \widehat Y^\pi - Y \|_{C([0,T])} 
   \le \sup_{i=0,\ldots,n} |\widehat{Y}_{t_i}^\pi-Y_{t_i}^\pi| +   \| Y^\pi -Y \|_{C([0,T])}, \]
where, as above,  
 \[ Y_t^\pi := (1-\theta) Y_{t_{i-1}} + \theta Y_{t_i}
   \sptext{1}{for}{1} t\in I_i^\pi \sptext{.5}{with}{.5}
   t=(1-\theta)t_{i-1} + \theta t_i. \]
The process $Y^\pi$ is a piece-wise linear and continuous process, but fails to be 
adapted in general. 
In this section we provide in Propositions \ref{prop:large_deviation_general_sigma} and \ref{prop:large_deviation_bounded_sigma} below
large deviation type estimates for $\|Y^\pi-Y\|_{C([0,T])}$.
We start with the following simple observation that links the distribution of the spline to our
results:

\begin{lemma}
\label{lemma:basic_spline_estimate}
Assume that there is a $\lambda_0\ge 0$ and a function $G:[\la_0,\infty)\times [0,T] \to [0,\infty)$,
non-decreasing in its second component,
such that 
\[ \P\left (\sup_{u\in [s,t]} \frac{|Y_u-Y_s|}{\sqrt{t-s}} > \la \right ) \le  G(\la,t-s)
   \sptext{1}{for}{1} \la\ge \la_0. \]
Then one has that
\[ \P\left (\sup_{t\in [0,T]} |Y_t-Y_t^\pi| > \la\right ) 
   \le n G\left (\frac{\la}{2\sqrt{|\pi|}},|\pi|\right )
   \sptext{1}{for}{1} \la \ge 2 \sqrt{|\pi|} \la_0.
\]
\end{lemma}

 \begin{proof}
We have that
\equa
\sup_{t\in [0,T]} |Y_t-Y_t^\pi| 
& = & \sup_{i=1,\ldots,n} \sup_{\theta\in [0,1]} \left |Y_{(1-\theta)t_{i-1}+\theta t_i} - ((1-\theta)Y_{t_{i-1}} + \theta Y_{t_i})\right |\\
&\le& \sup_{i=1,\ldots,n} \max\left \{ \sup_{t\in I_i^\pi} |Y_t-Y_{t_{i-1}}|, \sup_{t\in I_i^\pi} |Y_t-Y_{t_i}|\right \} \\
&\le& \sup_{i=1,\ldots,n} \left ( \sup_{t\in I_i^\pi} |Y_t-Y_{t_{i-1}}| + |t_i-t_{i-1}|\right ) \\
&\le& 2 \sup_{i=1,\ldots,n} \sup_{t\in I_i^\pi} |Y_t-Y_{t_{i-1}}|.
\tion
For $\la\ge 2 \la_0$ this implies our statement because
\equa
      \P\left (\sup_{t\in [0,T]} |Y_t-Y_t^\pi| > \sqrt{|\pi|}\la\right ) 
&\le& \sum_{i=1}^n  \P \left ( 2 \sup_{t\in I_i^\pi} |Y_t-Y_{t_{i-1}}| > \sqrt{|\pi|} \la \right )  \\
&\le& \sum_{i=1}^n  \P \left ( \sup_{t\in I_i^\pi} |Y_t-Y_{t_{i-1}}| > \frac{\sqrt{|I_i^\pi|} \la}{2} \right )  \\
&\le& \sum_{i=1}^n  G\left (\frac{\la}{2},|I_i^\pi|\right )  \\
&\le& n G\left (\frac{\la}{2},|\pi|\right). \qedhere
\tion
\end{proof}

In order to apply Theorem \ref{theorem:FBSDE_tail_Y} we let,
for $\lambda>0$ and $0\le s < t \le T$,
\equa
F(\lambda) &:= & \P\left (\sup_{u\in [0,T]} |X_u| > \lambda \right ),\\
G_\ell(\lambda) &:= & \inf \left \{ e^{-\mu} + F\left (\sqrt{\nu^2-1}\right ) : \lambda = \mu \nu \mbox{ with }\mu >0,\nu>1  \right \},\\
G_b(\lambda,t-s) &:= & \inf \left \{ e^{-\mu} + F\left ( \sqrt{ \frac{\nu^2 -1}{t-s} } \right ) : \lambda = \mu \nu \mbox{ with } \mu>0,\nu >1 \right \}.
\tion
The subscript $\ell$ stands for a linear growth of $\sigma$, the subscript $b$ for a bounded $\sigma$. 
For the function $F$ we get the following upper bounds:

\pagebreak
\begin{lemma}
\label{lemma:upper_bound_Psi}
\hfill
\begin{enumerate}
\item[$\mathbf{(i)}$] 
      Under the condition $(A_{b,\sigma})$ there exist $\alpha>0$ and $\la_0\ge 1$ depending at most on $(x,b,\sigma,T)$
      such that, for $\lambda\ge \lambda_0$,
      \[ F(\la) \le e^{-\alpha (\log \la)^2}. \]
\item[$\mathbf{(ii)}$] 
      Under the conditions $(A_{b,\sigma})$ and $(A_\sigma)$ there is a $c>0$ depending at most on $(x,b,\sigma,T)$
      such that, for $\lambda \ge  0$,
      \[ F(\la) \le c e^{-(\la/c)^2}. \]
\end{enumerate}
\end{lemma}

\begin{proof}
For $p\in [2,\infty)$ one has the estimates 
\[ \left \| \sup_{t\in [0,T]} | X_t| \right \|_p \le e^{cp}
    \sptext{1}{and}{1}
    \left \| \sup_{t\in [0,T]} | X_t| \right \|_p \le c \sqrt{p} \]
under $(A_{b,\sigma})$ and $(A_{b,\sigma},A_\sigma)$, respectively,
for constants $c>0$  depending  at most on $(x,b,\sigma,T)$.
Both estimates are known. They can be proved
by the standard Gronwall argument (cf. \cite[Lemma A.2]{Avikainen:arxiv:07}) but one has to use the estimate 
$\beta_p \le c \sqrt{p}$ for $p\in [2,\infty)$ from Proposition \ref{proposition:BDG}.
\medskip

$\mathbf{(i)}$ For all $\la>0$,
\[ \P(X_T^* >\la) \le \frac{1}{\la^p} \E |X_T^*|^p 
                  \le \frac{1}{\la^p} e^{cp^2}. \]
We set $\lambda_0 := e^{4c}$ and get for $\lambda \ge \lambda_0$ a $p\in [2,\infty)$ with 
$p=\frac{\log \la}{2c}$, so that   
$\frac{1}{\la^p} e^{cp^2}=e^{-\frac{(\log \la)^2}{4c}}$.
                  
$\mathbf{(ii)}$  Again, for all $\la>0$,
\[ \P(X_T^* >\la) \le \frac{1}{\la^p} \E |X_T^*|^p 
                  \le \frac{1}{\la^p} c^p p^{\frac{p}{2}}. \]
Assume $\la \ge \sqrt{2} c e$ and set $p:= (\la /(ce))^2 \in [2,\infty)$. Then
\[ \P(X_T^* >\la) \le e^{-\la^2/(c e)^2}. \]
Consequently, $\P(X_T^* >\la) \le e^{2-\la^2/(c e)^2}$ for all $\la\ge 0$.
\end{proof}

We derive the following bounds for $G_\ell$ and $G_b$:

\begin{lemma}\hfill
\label{lemma:upper_bound_Phi}

\begin{enumerate}
\item[$\mathbf{(i)}$] 
      Under the condition $(A_{b,\sigma})$ there exist $\alpha>0$ and $\la_0\ge 1$ depending at most on $(x,b,\sigma,T)$
      such that, for $\lambda\ge \lambda_0$,
      \[ G_\ell(\la) \le e^{-\alpha (\log \la)^2}. \]
\item [$\mathbf{(ii)}$] 
      Under the conditions $(A_{b,\sigma})$ and $(A_\sigma)$ there is a $c>0$ depending at most on $(x,b,\sigma,T)$
      such that, for $0\le s < t \le T$,
      \[ G_b(\la,t-s) \le 
         \begin{cases}c e^{-\frac{1}{c} \la}& 0<\la \le \frac{1}{t-s} \\
                     c e^{-\frac{1}{c} \la^\frac{2}{3} (t-s)^{-\frac{1}{3}}  }& \la > \frac{1}{t-s}   \\
                        \end{cases}. \]
\end{enumerate}
\end{lemma}

Under the conditions $(A_{b,\sigma},A_\sigma)$ we let $\la_0:=0$ and
$G_b(0,t-s):=\lim_{\la\downarrow 0} G_b(\la,t-s)$ so that
$G_b(0,t-s)\le c$.

\begin{proof}[Proof of Lemma \ref{lemma:upper_bound_Phi}]
For both cases we can replace $\nu>1$ in the definitions of $G_\ell$ and $G_b$ by $\nu \ge \sqrt{4/3}$ to replace the term 
$\sqrt{\nu^2-1}$ by $\nu/2$ to simplify the computation.
\smallskip

$\mathbf{(i)}$ We use the decomposition $\la=\mu\nu=\sqrt{\la}\sqrt{\la}$ and Lemma \ref{lemma:upper_bound_Psi}(i) (where $\alpha,\la_0>0$
might change).
\smallskip

$\mathbf{(ii)}$ In the case $\la \le \frac{1}{t-s}$ we use the decomposition $\mu =\sqrt{3/4} \lambda$ and 
$\nu = \sqrt{4/3}$, and in the case $\la > \frac{1}{t-s}$ we use $\mu =\sqrt{3/4} \lambda^{2/3}(t-s)^{-1/3}$ and $\nu = \sqrt{4/3}\la^{1/3} (t-s)^{1/3}$.
Then we apply Lemma \ref{lemma:upper_bound_Psi}(ii).
\end{proof}
\medskip

From Theorem \ref{theorem:FBSDE_tail_Y} we know that
\begin{equation}
\label{eqn:upper_bounds_Y_via_Phi}
\P \left ( \sup_{u\in [s,t]} \frac{|Y_u-Y_s|}{\sqrt{t-s}} > A_{\eqref{theorem:FBSDE_tail_Y}} \lambda \right )
   \le A \begin{cases}
         G_\ell(\lambda)     &: (A_{b,\sigma}) \\
         G_b(\lambda,t-s) &: (A_{b,\sigma},A_\sigma)
         \end{cases}
\end{equation}
for $\la \ge \la_0$ and $0\le s < t \le T$. Here $A:= c_0 \vee e$ with $c_0>0$ taken from Theorem \ref{theorem:FBSDE_tail_Y}, 
$A_{\eqref{theorem:FBSDE_tail_Y}}:= c_0 c_\eqref{theorem:main_FBSDE}$ in the case $(A_{b,\sigma})$,
and $A_{\eqref{theorem:FBSDE_tail_Y}}:= c_0 d_\eqref{theorem:main_FBSDE}$ in the case $(A_{b,\sigma},A_\sigma)$. 
To provide the large deviation type inequalities, we let
$\pi_n=(iT/n)_{i=0}^n$ be the equidistant net with $n+1$ knots and denote $Y^n:=Y^{\pi_n}$.
\medskip

\begin{proposition}
\label{prop:large_deviation_general_sigma}
Under the condition $(A_{b,\sigma})$ one has for $n\ge 2$ and $\la\ge \la_0$  that 
\[ \P\left (\| Y-Y^n \|_{C([0,T])} > \alpha_n \lambda\right )
\le 2 A e^{-\alpha  \left (\log \la \right )^2}.\] 
where $\alpha_n:= 2\sqrt{T} A_{\eqref{theorem:FBSDE_tail_Y}} n^{-1/2} e^{\sqrt{\frac{1}{\alpha}\log \frac{n}{2}}}$
and $\alpha,\la_0$ are taken from Lemma \ref{lemma:upper_bound_Phi}(i).

\end{proposition}

\begin{proof}
For $n\ge 2$ and $\la\ge \la_0$ we get from Lemmas \ref{lemma:basic_spline_estimate} and \ref{lemma:upper_bound_Phi}(i) that
\equa
      \P(\| Y-Y^n \|_{C([0,T])} > \alpha_n \lambda)
&\le& n A G_\ell\left ( \frac{\alpha_n \lambda \sqrt{n}}{2\sqrt{T} A_{\eqref{theorem:FBSDE_tail_Y}}}\right ) \\
& = & n A G_\ell\left ( \la e^{\sqrt{\frac{1}{\alpha} \log \frac{n}{2}}} \right ) \\
&\le& n A e^{-\alpha\left  (\log \left ( \la e^{\sqrt{\frac{1}{\alpha}\log \frac{n}{2}}}\right  ) \right )^2}\\
& = & n A e^{-\alpha \left ( \log  \la + \sqrt{\frac{1}{\alpha}\log \frac{n}{2}}\right )^2}\\
&\le& n A e^{-\alpha (\log  \la)^2 - \log \frac{n}{2}}\\
& = & 2 A e^{-\alpha (\log  \la)^2}.
\tion
\end{proof}

Using \eqref{eqn:upper_bounds_Y_via_Phi} for $(A_{b,\sigma},A_\sigma)$ 
gives the following large deviation estimate:

\begin{proposition}
\label{prop:large_deviation_bounded_sigma}

Under the conditions $(A_{b,\sigma})$ and $(A_\sigma)$ there is a constant $c>0$ such that 
\equa
\limsup_{\lambda \to \infty} \frac{\log \P(\|Y-Y^n\|_{C([0,T])} > \lambda)}{\lambda^\frac{2}{3}} &\le& - c n^\frac{2}{3}
       \sptext{1}{for}{1} n\ge 1,\\
\limsup_{n \to \infty} \frac{\log \P(\|Y-Y^n\|_{C([0,T])} > \lambda)}{n^\frac{1}{2}} &\le& - c \lambda
        \sptext{1.5}{for}{1}\lambda >0.
\tion
\end{proposition}

\begin{proof}
For the first inequality we use the case $\la > \frac{1}{t-s}=\frac{n}{T}$ in inequality \eqref{eqn:upper_bounds_Y_via_Phi},
for the second inequality the case $0<\la \le \frac{1}{t-s}=\frac{n}{T}$   in inequality \eqref{eqn:upper_bounds_Y_via_Phi}.
\end{proof}

\subsection{Confidence interval for direct simulation}
\label{sec:direct_simulation}

Assume that we are interested in the computation of $E(Y_t-Y_s)$ for fixed $0\le s < t \le T$ and that we can simulate independent copies
$D_1,\ldots,D_n$ of $Y_t-Y_s$.  Below we give an estimate 
on the confidence interval that is based on our tail-estimates. We start with a general lemma, that should be folklore.
To this end, assume an i.i.d. sequence of random variables $D_1,D_2,\ldots: \widehat\Omega\to \R$ such that 
$D_1 \in \bigcap_{p\in (0,\infty)} L_p(\widehat\Omega,\widehat\cF,\widehat\P)$, and let
\[ S_n := \frac{1}{n} (D_1+\cdots+D_n) 
   \sptext{1}{and}{1}
   \mu:=\E D_1. \] 
\begin{lemma}
\label{lemma:variance_reduction}

For $\vare>0$ one has 
   \[ \widehat\P(|S_n-\mu|>\vare)
       \le \inf_{p\in [2,\infty)} \left ( \frac{2(p-1)}{\sqrt{n} \vare} \| D_1 \|_p \right)^p. \]
 \end{lemma}
   
\begin{proof}
For $p\in [2,\infty)$ we have
\[
      \| S_n-\mu\|_p
 =  \frac{1}{n} \left \| \sum_{i=1}^n [D_i-\mu] \right \|_p 
\le  \frac{\beta_p}{n} \left \| \left ( \sum_{i=1}^n |D_i-\mu|^2 \right )^\frac{1}{2} \right \|_p 
\le \frac{\beta_p}{n} \left ( \sum_{i=1}^n \| D_i -\mu\|_p^2 \right )^\frac{1}{2} 
 =  \frac{\beta_p}{\sqrt{n}}  \| D_1 -\mu\|_p
\]
where from \cite[Theorem 3.3]{Burkholder:89} we know that we can take $c_p=p-1$. Therefore, for $\vare>0$,
\[ \widehat\P(|S_n-\mu|>\vare)
   \le \frac{1}{\vare^p} \| S_n-\mu\|_p^p
   \le \frac{1}{\vare^p} \left (\frac{p-1}{\sqrt{n}}\right )^p  \| D_1 -\mu\|_p^p
   \le \left ( \frac{2(p-1)}{\sqrt{n} \vare} \right )^p  \| D_1 \|_p^p. \]
 \end{proof}  
 
Now let us assume that condition $(A_{b,\sigma})$ is satisfied and fix $0\le s < t \le T$.
Let $S_n$ be a direct simulation of $Y_t-Y_s$. From Lemma \ref{lemma:upper_bound_Phi}(i) we can deduce 
 \[  \left \| \frac{Y_t-Y_s}{\sqrt{t-s}} \right \|_p \le  e^{cp} \]
 for some $c>0$ and all $p\in [2,\infty)$. By Lemma \ref{lemma:variance_reduction},
 \[ \widehat \P(|S_n-\mu|>\vare)
   \le \left (\frac{2(p-1)}{\sqrt{n} \vare}  \sqrt{t-s} e^{cp} \right )^p  
   \le \left (\alpha  \frac{\sqrt{t-s}}{\sqrt{n} \vare}  e^{\alpha p} \right )^p\] 
 for some $\alpha=\alpha(c)>0$ and all $p\in [2,\infty)$.
Define
$\psi(\delta):= \inf_{p \in [2,\infty)} \left (\alpha  \delta e^{\alpha p} \right )^p$
for $\delta>0$. Then
\[  \widehat \P(|S_n-\mu|>\vare) \le \psi \left ( \frac{\sqrt{t-s}}{\sqrt{n} \vare} \right ). \]
It is not difficult to check that
\[ \lim_{\delta \downarrow 0} \frac{\psi(\delta)}{\delta^M} = 0 \sptext{1}{for all}{1} M>0 \]
(consider $\delta\in (0,e^{-4\alpha})$ and choose $p\in [2,\infty)$ with $\delta=e^{-2\alpha p}$ so that
$\psi(\delta) \le (\alpha^2\delta)^{\frac{p}{2}}$).
For example, this implies
\[  \lim_{n\to \infty} n^M \psi \left ( \frac{\sqrt{t-s}}{\sqrt{n} \vare} \right ) = 0 \sptext{1}{for all}{1} M>0. \]

\subsection{Change of measure}
\label{sec:change_of_measure}
We describe a consequence of the BMO-estimates with respect to a change of the underlying measure. 
Let $\Q$ be a probability measure that is absolutely continuous with respect to $\P$ and such that  for 
$L:= d\Q/d\P>0$ there are $c>0$ and $v\in (1,\infty)$ such that
\[ \E \left [L^v\mitt \cF_\tau \right ] \le c^v  \left [ \E \left [L \mitt \cF_\tau\right ] \right ]^v, \]
for all stopping times $\tau:\Omega \to [0,T]$
(i.e. $\Q$ satisfies a reverse H\"older inequality with exponent $v$, cf. Definition \ref{definition:A_p}).
Assume a positive c\`adl\`ag and adapted process  $\Phi=(\Phi_t)_{t\in [0,T]}$, $p\in (0,\infty)$, and 
a continuous and adapted process  $A=(A_t)_{t\in [0,T]}$ with $A_0\equiv 0$  with 
$ \| A \|_{\BMO_p^\Phi} < \infty$ (see equation \eqref{eqn:definition_BMO_p_Phi}).
 Let $\tau:\Omega \to [0,T]$ be a stopping time,
$B\in \cF_\tau$ be of positive measure, and $\vare,\nu>0$. If $\Q_B$ is the normalized restriction of $\Q$ to $B$
and $1=\frac{1}{u}+\frac{1}{v}$, then
\equa
&   & \hspace*{-4em}
      \Q_B\left (|A_T-A_{\tau}|>(1+\vare)\nu\right ) \\
&\le& \Q_B\left (|A_T-A_{\tau}|>(1+\vare)\Phi_\tau\right ) + \Q_B(\Phi_\tau>\nu) \\
& = & \frac{1}{\Q(B)} \int_B 1_{\{ |A_T-A_{\tau}|>(1+\vare)\Phi_\tau\}}  L d \P + \Q_B(\Phi_\tau>\nu) \\
& = & \frac{1}{\Q(B)} \int_B \E \left [ 1_{\{ |A_T-A_{\tau}|>(1+\vare)\Phi_\tau\}}  L \mitt \cF_\tau \right ] d \P
      + \Q_B(\Phi_\tau>\nu) \\
&\le& \frac{1}{\Q(B)}
      \int_B \E \left [ 1_{\{ |A_T-A_{\tau}|>(1+\vare)\Phi_\tau\}} \mitt \cF_\tau \right ]^\frac{1}{u}
             \E \left [ L^v                                        \mitt \cF_\tau \right ]^\frac{1}{v} d \P + \Q_B(\Phi_\tau>\nu) \\
&\le& \frac{c}{\Q(B)}
      \int_B \E \left [ 1_{\{ |A_T-A_{\tau}|>(1+\vare)\Phi_\tau\}} \mitt \cF_\tau \right ]^\frac{1}{u}
             \E [L|\cF_\tau] d \P + \Q_B(\Phi_\tau>\nu) \\
& = & c
      \int_B \E \left [ 1_{\{ |A_T-A_{\tau}|>(1+\vare)\Phi_\tau\}} \mitt \cF_\tau \right ]^\frac{1}{u}
             d \Q_B + \Q_B(\Phi_\tau>\nu) \\
&\le& c\int_B \left [ \frac{1}{(1+\vare)^p} \| A \|_{\BMO_p^\Phi}^p \right ]^\frac{1}{u} d \Q_B 
      + \Q_B(\Phi_\tau>\nu) \\
& = & \frac{c}{(1+\vare)^\frac{p}{u}} \| A \|_{\BMO_p^\Phi}^\frac{p}{u}   
      + \Q_B(\Phi_\tau>\nu) \\
&\le& \frac{c}{(1+\vare)^\frac{p}{u}} \| A \|_{\BMO_p^\Phi}^\frac{p}{u}   
      + \Q_B\left (\sup_{u\in [\tau,T]}\Phi_u>\nu\right ).
\tion
As a consequence we can apply Theorem \ref{theorem:Stefan1}, but now for the measure $\Q$ instead of $\P$.
Let us come back to our setting and recall the inequality \eqref{eqn:global_BMO_esimate_Y}, i.e.
\[    \| (Y_t-Y_0)_{t\in [0,T]} \|_{\BMO_p^{\Phi}} \le c_{\eqref{theorem:main}}. \]
So we can apply this change of measure technique in our context. 
A careful investigation of local settings (i.e. the consideration of fixed general sub-intervals $[s,t]\subset [0,T]$)
is not yet done.

\subsection{Outlook}

The methodology to use weighted BMO spaces in stochastic problems, in order to replace $L_p$ spaces,
is exploited in the context of approximation problems for stochastic integrals in \cite{Geiss:wBMO} and 
in the context of variational problems for BSDEs in this article. The natural question is, to 
which other problems this general technique might be applied.
 A natural candidate for such a problem would be the investigation of existing 
approximation schemes for BSDEs from the literature
(for example, \cite{Bouchard-Elie-Touzi},\cite{Bouchard-Touzi}, \cite{Zhang:04}, \cite{Zhang:01}). 
It might be that the partial backward structure of these schemes helps to apply 
weighted BMO techniques where one could use existing $L_2$ results.


\section{\st Appendix A: General tools}
\label{section:appendix_general_tools}

The following lemmas have been used before, Lemma \ref{lemma:(ii)}
to justify assumption {\st (C6)}, and Lemma \ref{lemma:to_stopping_times}
in the proof of Theorem  \ref{theorem:main}.}
\medskip

\begin{lemma}\label{lemma:(ii)}
Assume that $\M$ is locally $\sigma$-compact.
Let $(f(x))_{x\in \M}$ be a continuous stochastic process defined on a probability space
$(\hat\Om,\hat\cF,\hat\P)$, such that $\E \sup_{x \in K_n} |f(x)| < \infty$ for all $n \in \N$.
If $\cG \subseteq \hat\cF$ is a $\sigma$-algebra, then there exists a unique\footnote{Unique up to indistinguishability.}
continuous stochastic process
$((\E^\cG f)(x))_{x \in \M} := (g(x))_{x \in \M}$ such that ${\st \hat{\P}}\left(\E^\cG(f(x)) = g(x)\right)=1$ for all $x \in \M$, and
such that $g(x)$ is $\cG$-measurable for every $x \in \M$.
\end{lemma}
\begin{proof}
\textbf{(i)} Let $K$ be one of the sets $K_n$ as in Definition \ref{definition:localcompact}, and consider $f$ as the Banach-space valued random variable $f:\hat\Om \to C(K)$,
where $C(K)$ is the space of continuous functions on $K$ equipped with the sup-norm. This space is separable,
so that applying \cite[Theorem V.1.4]{Diestel} and properties of the Bochner integral
we find a $g:\hat\Om \to C(K)$ with the required properties.
\smallskip

\textbf{(ii)} Defining $(g^{K_n}(x))_{x\in K_n}$ and $(g^{K_{n+1}}(x))_{x\in K_{n+1}}$ as in step \textbf{(i)},
we have that $g^{K_{n}}$ and $g^{K_{n+1}}$ are indistinguishable in $K_n$.
Hence, we can consistently define one process in $\bigcup_{n=1}^\infty \mathring{K_n} = \M$.
\end{proof}
\medskip

\begin{lemma}\label{lemma:to_stopping_times}
Let $0 \le s < t \le T$, and assume that $(\alpha_u)_{u \in [s,t]}$ is a process with c\`adl\`ag paths, and such that
$\E \sup_{r \in [s,t]} |\alpha_r| < \infty$.
If for all $u \in [s,t]$ we have
$
  \E^{\cF_u} |\alpha_u| \le \weight_u,
$
where $(\weight_{u})_{u \in [s,t]}$ is a supermartingale with c\`adl\`ag paths, then
$
  \E^{\cF_\tau} |\alpha_\tau| \le \weight_\tau
$
holds for all stopping times $\tau:\Om \to [s,t]$.
\end{lemma}
\begin{proof}
\textbf{(i)} Assume that $\tau:\Om \to \{ s_1,\dots,s_n \}$ is a stopping time for some $n \in \N$, $s\le s_1 \le \dots \le s_n \le t$.
We have for all $i=1,\dots,n$ that
$\E^{\ftn_{s_i}} |\alpha_{s_i}| \le \weight_{s_i}.$
Now we have for any $A \in \cF_\tau$ that
\begin{equation*}
\int_A |\alpha_\tau| d\P
=  \sum_{i=1}^n \int_{A\cap\{\tau=s_i\}} |\alpha_{s_i}| d\P
\le \sum_{i=1}^n \int_{A\cap\{\tau=s_i\}} \weight_{s_i} d\P
=  \int_A \weight_{\tau} d\P.
\end{equation*}
\textbf{(ii)} Let $\tau:\Om \to [s,t]$ be a stopping time, and let $(\tau_n)_{n \in \N}$ be a sequence of stopping times such that \linebreak
$\tau_n(\om) \downarrow \tau(\om)$ for all $\om \in \Om$, and $\tau_n:\Om \to [s,t]$ has a finite range.
By step \textbf{(i)} we know that for all {\st $n \ge 1$ we have}
\begin{equation}\label{equation:discreetti}
\E^{\ftn_{\tau_n}} |\alpha_{\tau_n}| \le \weight_{\tau_n}.
\end{equation}
Consider now the martingale
\begin{equation*} (M_r^n)_{r \in [s,t]} := \left(\E \left[|\alpha_{\tau_n}| \mitt \cF_r \right]\right)_{r \in [s,t]}. \end{equation*}
By optional stopping, and the fact that $\tau \le \tau_n \le t$ for all {\st $n \ge 1$}, we have
\begin{equation*} \E\left[ M_{\tau_n}^n \mitt \cF_\tau \right] = M_\tau^n. \end{equation*}
Moreover, using optional stopping and the fact that $\weight$ is a right-continuous supermartingale,
we deduce
\begin{equation*} \E^{\ftn_\tau} \weight_{\tau_n} \le \weight_{\tau}. \end{equation*}
Now, applying $\E^{\ftn_\tau}$ on both sides of equation (\ref{equation:discreetti}), we have that
\begin{equation*} \E^{\cF_\tau} |\alpha_{\tau_n}| \le \weight_{\tau}. \end{equation*}
Since $\alpha$ is right-continuous, we may apply dominated convergence to deduce that we have for any $A \in \ftn_\tau$
\begin{equation*}
\int_A \E\left[ |\alpha_\tau| \mitt \cF_\tau \right] d\P
=  \lim_n \int_A |\alpha_{\tau_n}| d\P
\le \lim_n \int_A \weight_{\tau} d\P.
\end{equation*}
\end{proof}


\section{\st Appendix B: Tools related to decoupling}
\label{section:appendix_tools_decoupling}

{\st The aim of the section is the proof of Proposition \ref{proposition:ugliness} below that 
was used in the proof of Corollary \ref{cor:equivalence_generator}.
We start with some preparations before we turn to  Proposition \ref{proposition:ugliness}.}
\medskip

Given a probability space $(\hat \Om,\hat \cF,\hat \P)$, the space of equivalence classes $L_0(\hat \Om,\hat \cF,\hat \P)$ can be equipped with the metric
 \begin{equation*}
  d(X,X') := \int_{\hat \Om} \frac{|X-X'|}{1+|X-X'|} d\hat \P.
 \end{equation*}
It is proven in \cite[\st Proposition 2.5]{jossain} that the decoupling operators defined in Section \ref{section:coupling} are isometries.
In particular, given a Borel-measurable function $\vph:(0,T]\to [0,1]$ and $S \in \{0,T\}$,
it follows for any \linebreak
$X,Y \in L_0(\overline{\Om}_S,\Sigma_S^0,\overline{\P}_S)$ that $d(X,Y)=d(X^\vph,Y^\vph)$.
\smallskip

\begin{lemma}\label{lemma:small_observation_new}
Assume that $\M$ is locally $\sigma$-compact, and let $\A \subseteq \M$ dense and countable.
If $h:\M\to\R$ is continuous,
then $\sup_{x \in \M} h(x) = \sup_{x \in \A} h(x)$. Furthermore,
if $f_1,f_2 \in [f] \in L_0(\overline{\Om}_T,{\st {\mathcal B}([0,T])\otimes \overline{\cF}},\overline{\P};C(\M))$,
then
 \begin{equation*}
  \E \int_0^T \sup_{x\in \M}|f_2(r,x)|dr = \E \int_0^T \sup_{x\in \M}|f_1(r,x)|dr.
 \end{equation*}
\end{lemma}

{\st For the following recall that $(\overline{\Om}_S,\Sigma_S^\varphi,\overline{\P}_S)$ was introduced in 
equation \eqref{eqn:definition_product_spaces}.}

\begin{lemma}\label{lemma:consistent_supremum}
Assume that $\M$ is locally $\sigma$-compact.
Let $S \in \{0,T\}$, 
$f \in \mathcal{L}_0(\overline{\Om}_S,\Sigma_S^0,\overline{\P}_S;C(\M))$, and put for all $\eta \in \overline{\Om}_S$ and all $x \in \M$
 \begin{equation*}
  g(\eta,x):= f(\eta,x)
              1_{\{\tilde\eta \in \overline{\Om}_S \mitt \sup_{y \in \M} f(\tilde\eta,y) \in \R\} }.
 \end{equation*}
Then it holds that $g \in \mathcal{L}_0(\overline{\Om}_S,\Sigma_S^0,\overline{\P}_S;C(\M))$, and any representative 
$g^\vph \in \mathcal{L}_0(\overline{\Om}_S,\Sigma_S^\vph,\overline{\P}_S;C(\M))$ satisfies $\overline{\P}_S(\sup_{x \in \M} g^\vph(x) \in \R)=1$, and
 \begin{equation*}
  \sup_{x \in \M} g^\vph(x) 1_{\{\sup_{y \in \M} g^\vph(y) \in \R\}} \in \left(\sup_{x \in \M} g(x)\right)^\vph.
 \end{equation*}
Consequently, there exists a representative $h^\vph$ of $g^\vph \in L_0(\overline{\Om}_S,\Sigma_S^\vph,\overline{\P}_S;C(\M))$ such that
        \begin{equation*}
         \sup_{x \in \M} h^\vph(x) \in \left(\sup_{x \in \M} g(x)\right)^\vph.
        \end{equation*}
\end{lemma}
\begin{proof}
The claim $g \in \mathcal{L}_0(\overline{\Om}_S,\Sigma_S^0,\overline{\P}_S;C(\M))$ follows from Lemma \ref{lemma:small_observation_new}.
Since $\sup_{x \in \A} g(\eta,x) \in \R$ for all $\eta \in \overline{\Om}_S$,
we have that
$
  \sup_{x \in \A} g(x) \in \mathcal{L}_0(\overline{\Om}_S,\Sigma_S^0,\overline{\P}_S).
$
Since $\A$ is countable, we can fix finite sets $\A_1\subseteq \A_2 \subseteq \dots \subseteq \A$ such that $\bigcup_{n \in \N} \A_n = \A$.
Using Proposition \ref{proposition:coupling_properties}(ii) and the isometry-property, we have
\begin{equation*}
    d\left( \left(\sup_{x \in \A} g(x)\right)^\vph , \sup_{x \in \A_k} g^\vph(x) \right)
= d\left( \left(\sup_{x \in \A} g(x)\right)^\vph , \left(\sup_{x \in \A_k} g(x)\right)^\vph \right)
= d\left( \sup_{x \in \A} g(x) , \sup_{x \in \A_k} g(x) \right) \to 0,
\end{equation*}
as $k \to \infty$. From this, and from the fact that $(\sup_{x \in \A_k} g^\vph(\eta,x))_{k \in \N}$ is monotone for all $\eta \in \overline{\Om}_S$, we deduce that
$\sup_{x \in \A_k} g^\vph(x)$ converges $\overline{\P}_S$-a.s. to $\left(\sup_{x \in \A} g(x)\right)^\vph$.
On the other hand, the monotonicity also implies that
 \begin{equation*}
  \lim_{k \to \infty} \sup_{x \in \A_k} g^\vph(\eta,x) = \sup_{x \in \A} g^\vph(\eta,x)
 \end{equation*}
for all $\eta \in \overline{\Om}_S$. Hence, it follows from continuity that $\sup_{x \in \M} g^\vph(x)$ is $\overline{\P}_S$-a.s. finite and
 \begin{equation*}
  \sup_{x \in \M} g^\vph(x) 1_{\{\sup_{y \in \M} g^\vph(y) \in \R \}} \in \left(\sup_{x \in \M} g(x)\right)^\vph.
 \end{equation*}
\end{proof}

\begin{remark}\rm
Lemma \ref{lemma:consistent_supremum} implies that if the assumptions of Lemma \ref{lemma:(ii)} are satisfied by \linebreak
$f \in \mathcal{L}_0(\overline{\Om}_T,\Sigma_T^0,\overline{\P}_T;C(\M))$,
then they are also satisfied by $f^\vph$.
This holds, since applying Lemma \ref{lemma:consistent_supremum} restricted to a compact $K \subseteq \M$,
we notice that if $\E \sup_{x \in K} |f(x)| < \infty$, then $\E \sup_{x \in K} |f^\vph(x)| = \E \sup_{x \in K} |f(x)|$.
\end{remark}
\smallskip

\begin{lemma}[{\cite[\st Remark 2.14]{jossain}}]
\label{lemma:old}
Let $X \in \mathcal{L}_0(\overline{\Om}_T,\Sigma_T^0,\overline{\P}_T)$ such that $\int_0^T |X(t,\om)| dt < \infty$ for all $\om \in \overline{\Om}$.
Then for any representative
$X^\vph \in \mathcal{L}_0(\overline{\Om}_T,\Sigma_T^\vph,\overline{\P}_T)$ we have that 
$\overline{\P}\left( \int_0^T |X^\vph(t)| dt < \infty\right)=1$, and
 \begin{equation*}
    \int_0^T X^\vph(t) 1_{\{ \int_0^T |X^\vph(s)| ds < \infty\}} dt \in \left( \int_0^T X(t) dt \right)^\vph.
 \end{equation*}
\end{lemma}
\smallskip

\begin{lemma}\label{corollary:x}
Let $\M$ be locally $\sigma$-compact and let $f \in \mathcal{L}_0(\overline{\Om}_T,\Sigma_T^0,\overline{\P}_T;C(\M))$ such that
 \begin{equation*}
  \overline{\P} \left( \int_0^T \sup_{x \in \M} |f(t,\om,x)| dt < \infty \right) = 1.
 \end{equation*}
Then there exists a representative $h^\vph$ of $|f^\vph| \in L_0(\overline{\Om}_T,\Sigma_T^\vph,\overline{\P}_T;C(\M))$ such that
 \begin{equation*}
        \int_0^T \sup_{x \in \M} |h^\vph(t,x)| dt
   \in   \left( \int_0^T \sup_{x \in \M} |f(t,x)| dt \right)^\vph.
 \end{equation*}
\end{lemma}
\begin{proof}
First note that
 \begin{equation*}
  \overline{\P} \left( \int_0^T \sup_{x \in \M} |f(r,\om,x)| dr < \infty \right) = 1
 \end{equation*}
implies
 \begin{equation*}
  \overline{\P}_T \left( \sup_{x \in \M} |f(t,\om,x)| < \infty \right) = 1.
 \end{equation*}
We may redefine $f$ such that $\sup_{x \in \M} |f(t,\om,x)| < \infty$ for all $(t,\om) \in \overline{\Om}_T$,
and $\int_0^T \sup_{x \in \M} |f(r,\om,x)| dr < \infty$ for all $\om \in \overline{\Om}$.
It is a direct consequence of Proposition \ref{proposition:coupling_properties}(ii) that
$|f|^\vph = |f^\vph|$, so that we may look for a representative
of $|f|^\vph$ that satisfies the claim.
Applying Lemma \ref{lemma:consistent_supremum} to $|f|$ gives us a representative $h^\vph$ of $|f|^\vph \in L_0(\overline{\Om}_T,\Sigma_T^\vph,\overline{\P}_T;C(\M))$
such that $\sup_{x \in \M} h^\vph(x) \in (\sup_{x \in \M}|f(x)|)^\vph$.
Letting $X(t,\om):=\sup_{x \in \M} |f(t,\om,x)|$ for $(t,\om) \in \overline{\Om}_T$,
we then have that $\sup_{x \in \M} h^\vph(x)$ is a representative of $X^\vph$.
Hence, Lemma \ref{lemma:old} implies that
 \begin{equation*}
  \overline{\P} \left( \int_0^T \sup_{x \in \M}|h^\vph(t,x)| dt < \infty \right) = 1,
 \end{equation*}
and
 \begin{equation*}
    \int_0^T \sup_{x \in \M}|h^\vph(t,x)| 1_{\{ \int_0^T \sup_{x \in \M}|h^\vph(r,x)| dr < \infty\}} dt
    \in \left( \int_0^T \sup_{x \in \M} |f(t,x)| dt \right)^\vph.
 \end{equation*}
The representative of $|f|^\vph \in L_0(\overline{\Om}_T,\Sigma_T^\vph,\overline{\P}_T;C(\M))$ that satisfies the claim, is
$|h^\vph| 1_{\{ \int_0^T \sup_{x \in \M}|h^\vph(r,x)| dr < \infty\}}$.
\end{proof}

We are ready to prove the desired result. {\st Recall that
$\overline{\cG_s^t}$ and $\overline{\cH_s^t}$ were defined in \eqref{eqn:sigma_algebras_bar}
and $\xi^{(a,b]}$ in \eqref{eqn:xiab}.}

\begin{proposition}\label{proposition:ugliness}
Assume that $\M$ is locally $\sigma$-compact.
Let {\st $p\in [1,\infty)$}, 
$0 \le s < t \le T$, $0 \le u_1 < u_2 \le T$, and $f \in \mathcal{L}_0(\overline{\Om}_T,\Sigma_T^0,\overline{\P}_T;C(\M))$ such that
$\int_{\overline{\Om}_T} \sup_{x \in K} |f(x)| d\overline{\P}_T < \infty$ for every compact $K \subseteq \M$.
If
 \begin{equation*}
  \left\|\int_{u_1}^{u_2} \sup_{x \in \M} |f(r,x)-(\E^{\overline{\cH_s^t}} f)(r,x)| dr \right\|_p < \infty,
 \end{equation*}
then $\overline{\P}$-a.s.
\begin{equation*}
 \E^{\overline{\cG_s^t}} \left(\int_{u_1}^{u_2} \sup_{x\in \M} |f(r,x)-f^{(s,t]}(r,x)| dr \right)^p
\le 2^p \E^{\overline{\cG_s^t}} \left(\int_{u_1}^{u_2} \sup_{x\in \M} |f(r,x)-(\E^{\overline{\cH_s^t}} f)(r,x)| dr \right)^p.
\end{equation*}
Conversely, if 
 \begin{equation*}
  \left\|\int_{u_1}^{u_2} \sup_{x \in \M} |f(r,x)-f^{(s,t]}(r,x)| dr \right\|_p < \infty,
 \end{equation*}
then $\overline{\P}$-a.s.
\begin{equation*}
    \E^{\overline{\cG_s^t}} \left(\int_{u_1}^{u_2} \sup_{x\in \M} |f(r,x)-(\E^{\overline{\cH_s^t}} f)(r,x)| dr \right)^p
\le \E^{\overline{\cG_s^t}} \left(\int_{u_1}^{u_2} \sup_{x\in \M} |f(r,x)-f^{(s,t]}(r,x)| dr \right)^p.
\end{equation*}
\end{proposition}

\begin{remark}\rm
$ $
\begin{enumerate}
\item[(1)] To define $\E^{\overline{\cH}_s^t}f \in
           \mathcal{L}_0(\overline{\Om}_T,\overline{\cH_s^t},\overline{\P}_T;C(\M))$ we apply Lemma \ref{lemma:(ii)}, and this is why we need to assume that $\int_{\overline{\Om}_T} \sup_{x \in K} |f(x)| d\overline{\P}_T < \infty$ for every compact $K \subseteq \M$.
\item[(2)] The conclusion of Proposition \ref{proposition:ugliness}
               with $p=1$ implies that
                \begin{equation*}
                   \frac{1}{2} \E^{\overline{\cH_s^t}} \| f - f^{(s,t]} \|_{C(\M)}
                 \le  \E^{\overline{\cH_s^t}} \| f - 
                      \E^{\overline{\cH_s^t}} f \|_{C(\M)}
                 \le  \E^{\overline{\cH_s^t}} \| f - f^{(s,t]}
                       \|_{C(\M)}.
                \end{equation*}
                   Hence, Proposition \ref{proposition:ugliness} generalizes Lemma \ref{lemma:conditional_equivalence} from random variables $\xi:\Om \to \R$
                   to function-space valued stochastic processes $f:\Om_T \to C(\M)$.
\end{enumerate}
\end{remark}

\begin{proof}[{\st Proof of Proposition \ref{proposition:ugliness}}]
We will use $\|\cdot\|_p$ for $\|\cdot\|_{L_p(\overline{\Om})}$ and $\A$ for a fixed dense countable subset of $\M$.
Note that 
$\sup_{x \in \M} h(x)= \sup_{x \in \A} h(x)$ whenever $h$ is continuous, so that we may replace $\M$ by $\A$ in the proof below.
To simplify the notation in the proof, we assume that $u_1=0$ and $u_2=T$.
\smallskip

\textbf{Step 1:} We will first show that if $g \in \mathcal{L}_0(\overline{\Om}_T,\overline{\cH_s^t},\overline{\P}_T;C(\M))$
is such that
 \begin{equation*}
  \left\|\int_0^T \sup_{x \in \A} |g(r,x) - f(r,x)| dr \right\|_p < \infty,
 \end{equation*}
then
\begin{equation*}
      \left\|\int_0^T \sup_{x \in \A} |f(r,x) - f^{(s,t]}(r,x)|dr \right\|_p
\le  2\left\|\int_0^T \sup_{x \in \A} |f(r,x) - g(r,x)        | dr \right\|_p.
\end{equation*}
Fixing $g$ as described above,
Lemma \ref{lemma:adapted_extended}(v) implies that $g \in g^{(s,t]}$,
so Lemma \ref{corollary:x} applied to $g-f$
in particular implies
 \begin{equation*}
  \E \left( \int_0^T \sup_{x \in \A} |g(r,x) - f^{(s,t]}(r,x)| dr \right)^p
= \E \left( \int_0^T \sup_{x \in \A} |g(r,x) - f(r,x)| dr         \right)^p.
 \end{equation*}
From this we deduce
 \equa
&& \left\|\int_0^T \sup_{x \in \A} |f(r,x) - f^{(s,t]}(r,x)|dr \right\|_p \\
&\le& \left\|\int_0^T \sup_{x \in \A} |f(r,x) - g(r,x)        | dr \right\|_p
    + \left\|\int_0^T \sup_{x \in \A} |g(r,x) - f^{(s,t]}(r,x)| dr \right\|_p \\
&=&  2\left\|\int_0^T \sup_{x \in \A} |f(r,x) - g(r,x)        | dr \right\|_p.
 \tion

\textbf{Step 2:} We assume that $\|\int_0^T \sup_{x \in \A} |f(r,x)-f^{(s,t]}(r,x)|dr \|_p< \infty$, and will show that
\begin{equation*}
  \left\|\int_0^T \sup_{x \in \A} |f(r,x)-(\E^{\overline{\cH_s^t}}f)(r,x)|dr \right\|_p
\le \left\|\int_0^T \sup_{x \in \A} |f(r,x)-f^{(s,t]}(r,x)| dr \right\|_p.
\end{equation*}
We use $W^0,W^1$ to denote the canonical extensions of $W,W'$, respectively, and for
$0 \le a < b \le T$ we work with the $\sigma$-algebras

 \equa
  \cH_{a,b}^{W^0} &:=& \mathcal{B}([0,T]) \otimes \sigma( W_r^0-W_a^0, r \in [a,b]), \\
  \cH_{a,b}^{W^1} &:=& \mathcal{B}([0,T]) \otimes \sigma( W_r^1-W_a^1, r \in [a,b]), \\
  \mathcal{H}     &:=& \{\emptyset,[0,T]\} \otimes \sigma(W_r^0-W_s^0,r \in [s,t]).
 \tion

Note that these are $\sigma$-algebras in $\overline{\Om}_T$, and we have the inclusions
 \equa
   \cH_{0,T}^{W^0}						                    &\subseteq& \Sigma_T^0, \\
   \cH_{0,s}^{W^0}\vee\cH_{s,t}^{W^1}\vee\cH_{t,T}^{W^0}	&\subseteq& \Sigma_T^{(s,t]}, \\
   \cH_{0,s}^{W^0}\vee\cH_{t,T}^{W^0}                       &\subseteq& \overline{\cH_s^t}.
 \tion
Moreover, the inclusions are "up to nullsets", which in this context means that we have for example
 \begin{equation*}
  \Sigma_T^0 = \cH_{0,T}^{W^0} \vee
        \left( \mathcal{B}([0,T]) \otimes \overline{\cN} \right),
 \end{equation*}
where $\overline{\cN}$ are the $\overline{\P}$-nullsets. From this it follows that $f$ and $\E^{\cH_{0,T}^{W^0}}f$ given by Lemma \ref{lemma:(ii)} are indistinguishable. 
To keep the notations as light as possible, we simply say that using Lemma \ref{lemma:small_observation_new} and Lemma \ref{lemma:(ii)}, we may assume that
\begin{enumerate}[(1)]
 \item $f$ is $\cH_{0,T}^{W^0}$-measurable,
 \item $f^{(s,t]}$ is $\cH_{0,s}^{W^0}\vee\cH_{s,t}^{W^1}\vee\cH_{t,T}^{W^0}$-measurable,
 \item $\E^{\overline{\cH_s^t}}f^{(s,t]}$ is $\cH_{0,s}^{W^0}\vee\cH_{t,T}^{W^0}$-measurable.
\end{enumerate}
Then the facts that for all $x \in \A$
 \begin{equation*} \cH_{0,s}^{W^0}\vee\cH_{t,T}^{W^0}\vee\sigma(f^{(s,t]}(x)) \mbox{ is independent of } \mathcal{H}, \end{equation*}
 \begin{equation*} f(x) \mbox{ is } \cH_{0,s}^{W^0}\vee\cH_{t,T}^{W^0}\vee\mathcal{H} \mbox{-measurable}, \end{equation*}
are immediate. Hence, it follows from \cite[9.7(k)]{Williams} that
 \begin{equation*}
  \E^{\overline{\cH_s^t} \vee \cH} f^{(s,t]}(x) =
  \E^{\overline{\cH_s^t}} f^{(s,t]}(x)
 \end{equation*}
for all $x \in \A$.
Since $f(x) \in L_1(\overline{\Om}_T,\Sigma_T^0,\overline{\P}_T)$ for all $x \in \A$, it follows from Proposition \ref{lemma:adapted_extended}(i) that\linebreak
$\E^{\overline{\cH_s^t}} f(x) = \E^{\overline{\cH_s^t}} f^{(s,t]}(x)$ for all $x \in \A$.
Thus we have
 \equa
\left\|\int_0^T \sup_{x \in \A} |f(r,x)-(\E^{\overline{\cH_s^t}}f)(r,x)|dr \right\|_p
&=& \left\|\int_0^T \sup_{x \in \A}
               |f(r,x)-(\E^{\overline{\cH_s^t}}f^{(s,t]}(r,x))|dr \right\|_p \\
&=& \left\|\int_0^T \sup_{x \in \A}
               |f(r,x)-(\E^{\overline{\cH_s^t}\vee\mathcal{H}}f^{(s,t]}(r,x))|dr \right\|_p \\
&=& \left\|\int_0^T \sup_{x \in \A}
               |\E^{\overline{\cH_s^t}\vee\mathcal{H}}(f(r,x)-f^{(s,t]}(r,x))|dr \right\|_p \\
&\le& \left\|\int_0^T \sup_{x \in \A}
               \E^{\overline{\cH_s^t}\vee\mathcal{H}}|f(r,x)-f^{(s,t]}(r,x)|dr \right\|_p \\
&\le& \left\|\int_0^T \E^{\overline{\cH_s^t}\vee\mathcal{H}}
              \left( \sup_{x \in \A} |f(r,x)-f^{(s,t]}(r,x)| \right) dr \right\|_p \\
&=& \left\|\E^{\cF^0} \left( \int_0^T
               \sup_{x \in \A} |f(r,x)-f^{(s,t]}(r,x)| dr \right) \right\|_p \\
&\le& \left\|\int_0^T \sup_{x \in \A} |f(r,x)-f^{(s,t]}(r,x)| dr \right\|_p.
 \tion

\textbf{Step 3:} The conditional claim follows from the result with the full expectation as in Lemma \ref{lemma:conditional_equivalence}:
assume that $f \in \mathcal{L}_0(\overline{\Om}_T,\Sigma_T^0,\overline{\P}_T;C(\M))$ is such that
$\int_{\overline{\Om}_T} \sup_{x \in K} |f(x)| d\overline{\P}_T < \infty$
for every compact $K \subseteq \M$. Let $B \in \overline{\cG_s^t}$ with $\overline{\P}(B)>0$, and define
 \begin{equation*}
  \tilde f \in \mathcal{L}_0(\overline{\Om}_T,\Sigma_T^0,\overline{\P}_T;C(\M)) \, \text{ by } \,
  \tilde f(r,\om,x) := f(r,\om,x) 1_B(\om).
 \end{equation*}
Fixing any representative of $f^{(s,t]}$, we have that
\begin{enumerate}
\item[(1)] $\int_{\overline{\Om}_T} \sup_{x \in K} |\tilde f(x)| d\overline{\P}_T < \infty$ for every compact $K \subseteq \M$,
\item[(2)] $1_B (\E^{\overline{\cH_s^t}} f)$
       is a representative of $(\E^{\overline{\cH_s^t}} \tilde f)$,
 \item[(3)] $1_B f^{(s,t]}$ 
       is a representative of $\tilde f^{(s,t]}$,
\end{enumerate}
so the claim follows by applying steps 1 and 2 with $\tilde f$.
\end{proof}


\section{\st Appendic C: A John-Nirenberg type theorem}
\label{sec:JN}

{\st We recall the} result \cite{Geiss:wBMO} (Theorem \ref{theorem:Stefan1}). 
{\st Whereas} in \cite{Geiss:wBMO} c\`adl\`ag processes are considered, {\st we only need the case of continuous processes}.
Fix $R>0$, let $(\Om,\cG_R,\P,(\cG_r)_{r \in [0,R]})$ {\st be} a stochastic basis 
{\st such that $(\Om,\cG_R,\P)$ is complete, $(\cG_r)_{r \in [0,R]}$ is right-continuous, and
$\cG_0$ contains all nullsets,} and let 
$A = (A_r)_{r \in [0,R]}$ {\st be} a \emph{continuous},
adapted stochastic process with $A_0{\st \equiv}0$.
Moreover, we assume that $(\Psi_r)_{r \in [0,R]}$ is a c\`adl\`ag $(\cG_r)_{r \in [0,R]}$-adapted stochastic process
such that $\Psi_r(\om)>0$ for all $(r,\om)\in \Om_R$. Put
 \begin{equation*}
  \mathcal{S}^\cG_{0,R} := \left\{ \tau:\Om \to [0,R] \mitt \tau \text{ is a } (\cG_r)_{r \in [0,R]}-\text{stopping time} \right\}
 \end{equation*}
and define
 \begin{equation*} W_\Psi(B,\nu;\tau) := \P\left(B \cap \left\{ \sup_{u \in [\tau,R]}\Psi_{u} > \nu \right\}\right) \end{equation*}
for $\nu>0$, $\tau \in \mathcal{S}^\cG_{0,R}$, and $B \in \gtn_\tau$.
Recall that for $B \in \cG_R$ of positive measure
 \begin{equation*} \P_B \left( \cdot \right) := \frac{\P(B \cap \cdot)}{\P(B)}. \end{equation*}

\begin{theorem}[{\cite[Theorem 1]{Geiss:wBMO}}]\label{theorem:Stefan1}
Assume that there is an $\thteta \in (0,\half)$ such that
 \begin{equation}\label{equation:(1)}
  \P_B(|A_R-A_\tau|>\nu) \le \thteta + \frac{W_\Psi(B,\nu;\tau)}{\P(B)}
 \end{equation}
for all $\nu>0$, $\tau \in \mathcal{S}^\cG_{0,R}$, and $B \in \cG_\tau$ of positive measure.
 Then there are constants $a,c>0$, depending on $\thteta$ only, such that
\begin{equation*}
  \P_B \left( \sup_{u \in [\tau,R]}|A_u-A_\tau| > \lambda + a\mu\nu \right)
\le e^{1-\mu} \P_B\left( \sup_{u \in [\tau,R]}|A_u-A_\tau| > \lambda \right)
    +  c \frac{W_\Psi(B,\nu;\tau)}{\P(B)}
\end{equation*}
for all $\lambda,\mu,\nu>0$, $\tau \in \mathcal{S}^\cG_{0,R}$, and $B \in \cG_\tau$ of positive measure.
\end{theorem}


\section*{\st References}
\bibliographystyle{elsarticle-num}

\end{document}